\documentclass[11pt,oneside]{amsart}

\usepackage{fullpage}

\usepackage{amsthm,amsfonts,amssymb,amsmath,amsxtra}
\usepackage[all]{xy}
\SelectTips{cm}{}
\usepackage{tikz}
\usepackage{verbatim}

 \usepackage[margin=1.25in]{geometry}
\usepackage{mathrsfs}

\RequirePackage{xspace}
\RequirePackage{etoolbox}
\RequirePackage{varwidth}
\RequirePackage{enumitem}
\RequirePackage{tensor}
\RequirePackage{mathtools}
\RequirePackage{longtable}
\RequirePackage{multirow}
\RequirePackage{makecell}
\RequirePackage{bigints}

\usepackage[backref=page]{hyperref}

\newcommand{\BC}{\ensuremath{\mathbb{C}}\xspace}

\newcommand{\BE}{\ensuremath{\mathbb{E}}\xspace}

\newcommand{\BN}{\ensuremath{\mathbb{N}}\xspace}

\newcommand{\BQ}{\ensuremath{\mathbb{Q}}\xspace}

\newcommand{\BT}{\ensuremath{\mathbb{T}}\xspace}

\newcommand{\BV}{\ensuremath{\mathbb{V}}\xspace}

\newcommand{\BX}{\ensuremath{\mathbb{X}}\xspace}
\newcommand{\BY}{\ensuremath{\mathbb{Y}}\xspace}
\newcommand{\BZ}{\ensuremath{\mathbb{Z}}\xspace}

\newcommand{\CC}{\ensuremath{\mathcal{C}}\xspace}

\newcommand{\CE}{\ensuremath{\mathcal{E}}\xspace}
\newcommand{\CF}{\ensuremath{\mathcal{F}}\xspace}

\newcommand{\CH}{\ensuremath{\mathcal{H}}\xspace}

\newcommand{\CM}{\ensuremath{\mathcal{M}}\xspace}
\newcommand{\CN}{\ensuremath{\mathcal{N}}\xspace}
\newcommand{\CO}{\ensuremath{\mathcal{O}}\xspace}

\newcommand{\CS}{\ensuremath{\mathcal{S}}\xspace}
\newcommand{\CT}{\ensuremath{\mathcal{T}}\xspace}
\newcommand{\CU}{\ensuremath{\mathcal{U}}\xspace}

\newcommand{\CX}{\ensuremath{\mathcal{X}}\xspace}
\newcommand{\CY}{\ensuremath{\mathcal{Y}}\xspace}
\newcommand{\CZ}{\ensuremath{\mathcal{Z}}\xspace}

\newcommand{\RN}{\ensuremath{\mathrm{N}}\xspace}

\newcommand{\RU}{\ensuremath{\mathrm{U}}\xspace}

\DeclareMathOperator{\Aut}{Aut}

\newcommand{\del}{\operatorname{\partial Orb}}

\DeclareMathOperator{\diag}{diag}

\DeclareMathOperator{\End}{End}

\DeclareMathOperator{\Gal}{Gal}

\newcommand{\GL}{\mathrm{GL}}

\newcommand{\GU}{\mathrm{GU}}

\DeclareMathOperator{\Hom}{Hom}

\newcommand{\id}{\ensuremath{\mathrm{id}}\xspace}

\DeclareMathOperator{\im}{im}

\DeclareMathOperator{\Int}{\ensuremath{\mathrm{Int}}\xspace}

\DeclareMathOperator{\Lie}{Lie}

\newcommand{\M}{\mathrm{M}}

\DeclareMathOperator{\Nm}{Nm}

\DeclareMathOperator{\Orb}{Orb}

\newcommand{\red}{\ensuremath{\mathrm{red}}\xspace}

\DeclareMathOperator{\Res}{Res}
\newcommand{\rs}{\ensuremath{\mathrm{rs}}\xspace}

\DeclareMathOperator{\Spec}{Spec}
\DeclareMathOperator{\Spf}{Spf}

\newcommand{\val}{{\mathrm{val}}}

\DeclareMathOperator{\supp}{supp}

\newcommand{\U}{\mathrm{U}}

\DeclareMathOperator{\vol}{vol}

\newcommand{\wit}{\widetilde}
\newcommand{\wh}{\widehat}

\newcommand{\pair}[1]{\langle {#1} \rangle}

\newcommand{\ov}{\overline}

\newcommand{\lra}{\longrightarrow}
\newcommand{\imp}{\Longrightarrow}

\newcommand{\bs}{\backslash}

\newcommand{\lv}{\lvert}
\newcommand{\rv}{\rvert}

\newenvironment{altenumerate}
   {\begin{list}
      {(\theenumi) }
      {\usecounter{enumi}
       \setlength{\labelwidth}{0pt}
       \setlength{\labelsep}{0pt}
       \setlength{\leftmargin}{0pt}
       \setlength{\itemsep}{\the\smallskipamount}
       \renewcommand{\theenumi}{\roman{enumi}}
      }}
   {\end{list}}

\newenvironment{altitemize}
   {\begin{list}
      {$\bullet$}
      {\setlength{\labelwidth}{0pt}
	   \setlength{\itemindent}{5pt}
       \setlength{\labelsep}{5pt}
       \setlength{\leftmargin}{0pt}
       \setlength{\itemsep}{\the\smallskipamount}
      }}
   {\end{list}}

\setitemize[0]{leftmargin=*,itemsep=\the\smallskipamount}
\setenumerate[0]{leftmargin=*,itemsep=\the\smallskipamount}

\renewcommand{\to}{%
   \ifbool{@display}{\longrightarrow}{\rightarrow}%
   }
\let\shortmapsto\mapsto
\renewcommand{\mapsto}{%
   \ifbool{@display}{\longmapsto}{\shortmapsto}%
   }
\newlength{\olen}
\newlength{\ulen}
\newlength{\xlen}
\newcommand{\xra}[2][]{%
   \ifbool{@display}%
      {\settowidth{\olen}{$\overset{#2}{\longrightarrow}$}%
       \settowidth{\ulen}{$\underset{#1}{\longrightarrow}$}%
       \settowidth{\xlen}{$\xrightarrow[#1]{#2}$}%
       \ifdimgreater{\olen}{\xlen}%
          {\underset{#1}{\overset{#2}{\longrightarrow}}}%
          {\ifdimgreater{\ulen}{\xlen}%
             {\underset{#1}{\overset{#2}{\longrightarrow}}}
             {\xrightarrow[#1]{#2}}}}%
      {\xrightarrow[#1]{#2}}
   }
\makeatother
\newcommand{\xyra}[2][]{%
   \settowidth{\xlen}{$\xrightarrow[#1]{#2}$}%
   \ifbool{@display}%
      {\settowidth{\olen}{$\overset{#2}{\longrightarrow}$}%
       \settowidth{\ulen}{$\underset{#1}{\longrightarrow}$}%
       \ifdimgreater{\olen}{\xlen}%
          {\mathrel{\xymatrix@M=.12ex@C=3.2ex{\ar[r]^-{#2}_-{#1} &}}}%
          {\ifdimgreater{\ulen}{\xlen}%
             {\mathrel{\xymatrix@M=.12ex@C=3.2ex{\ar[r]^-{#2}_-{#1} &}}}
             {\mathrel{\xymatrix@M=.12ex@C=\the\xlen{\ar[r]^-{#2}_-{#1} &}}}}}%
      {\mathrel{\xymatrix@M=.12ex@C=\the\xlen{\ar[r]^-{#2}_-{#1} &}}}%
   }
\makeatletter
\newcommand{\xla}[2][]{%
   \ifbool{@display}%
      {\settowidth{\olen}{$\overset{#2}{\longleftarrow}$}%
       \settowidth{\ulen}{$\underset{#1}{\longleftarrow}$}%
       \settowidth{\xlen}{$\xleftarrow[#1]{#2}$}%
       \ifdimgreater{\olen}{\xlen}%
          {\underset{#1}{\overset{#2}{\longleftarrow}}}%
          {\ifdimgreater{\ulen}{\xlen}%
             {\underset{#1}{\overset{#2}{\longleftarrow}}}
             {\xleftarrow[#1]{#2}}}}%
      {\xleftarrow[#1]{#2}}
   }
\newcommand{\isoarrow}{%
   \ifbool{@display}{\overset{\sim}{\longrightarrow}}{\xrightarrow\sim}%
   }
\renewcommand{\lra}{%
   \ifbool{@display}{\longleftrightarrow}{\leftrightarrow}%
   }

\newcommand{\barE}{{\ov\BE}}
\newcommand{\Fb}{{\breve F}}
\newcommand{\OFb}{{O_{\breve F}}}

\newcommand{\rd}{\mathrm{d}}

\newcommand{\wt}{\wit}

\newcommand{\ch}{\mathrm{ch}}

\DeclareFontFamily{U}{matha}{\hyphenchar\font45}
\DeclareFontShape{U}{matha}{m}{n}{
      <5> <6> <7> <8> <9> <10> gen * matha
      <10.95> matha10 <12> <14.4> <17.28> <20.74> <24.88> matha12
      }{}
\DeclareSymbolFont{matha}{U}{matha}{m}{n}
\DeclareFontFamily{U}{mathx}{\hyphenchar\font45}
\DeclareFontShape{U}{mathx}{m}{n}{
      <5> <6> <7> <8> <9> <10>
      <10.95> <12> <14.4> <17.28> <20.74> <24.88>
      mathx10
      }{}
\DeclareSymbolFont{mathx}{U}{mathx}{m}{n}

\DeclareMathSymbol{\obot}         {2}{matha}{"6B}

\newtheorem{theorem}[subsubsection]{Theorem}
\newtheorem{proposition}[subsubsection]{Proposition}
\newtheorem{lemma}[subsubsection]{Lemma}
\newtheorem {conjecture}[subsubsection]{Conjecture}
\newtheorem{corollary}[subsubsection]{Corollary}

\theoremstyle{definition}
\newtheorem{definition}[subsubsection]{Definition}
\newtheorem{example}[subsubsection]{Example}

\newtheorem{remark}[subsubsection]{Remark}

\numberwithin{equation}{subsection}

\newcommand{\Z}{\mathcal{Z}}

\usepackage{xcolor}

\newcommand{\Nn}{\mathcal{N}_{n}}
\newcommand{\N}{\mathcal{N}_{n+1}}

\newcommand{\n}{\mathcal{N}_{n}}

\newcommand{\fp}{\varphi}

\newcommand{\kb}{\bar k}

\newcommand{\T}{\mathcal{T}}

\renewcommand{\k}{K}
\newcommand{\kf}{K^\flat}

\newcommand{\K}{{K'}}

\newcommand{\Kf}{{K'}^\flat}
\newcommand{\hkf}{\mathcal{H}_{\kf}}
\newcommand{\hk}{\mathcal{H}_\k}
\newcommand{\Hkf}{\mathcal{H}_{K^{\prime \flat}}}
\newcommand{\Hk}{\mathcal{H}_{K'}}

\newcommand{\Bc}{\mathrm{BC}}
\newcommand{\Sat}{\mathrm{Sat}}
\newcommand{\si}{\mathfrak{s}}
\newcommand{\sig}{\sigma}
\newcommand{\qbin}[2]{\left[\begin{smallmatrix}#1\\ #2\end{smallmatrix}\right]_{-q}}

\title[AFL for Hecke]{Arithmetic Fundamental Lemma for the spherical Hecke algebra}

\author{Chao Li}
\address{Columbia University, Department of Mathematics, 2990 Broadway,	New York, NY 10027, USA}
\email{chaoli@math.columbia.edu} 
\author{Michael Rapoport}
\address{Mathematisches Institut der Universit\"at Bonn, Endenicher Allee 60, 53115 Bonn, Germany, and University of Maryland, Department of Mathematics, College Park, MD 20742, USA}
\email{rapoport@math.uni-bonn.de}
\author{Wei Zhang}
\address{Massachusetts Institute of Technology, Department of Mathematics, 77 Massachusetts Avenue, Cambridge, MA 02139, USA}
\email{weizhang@mit.edu}

 \date{\today}

\begin{document}

\begin{abstract}
We define Hecke correspondences and Hecke operators on unitary RZ spaces and study their basic geometric properties,  including a commutativity conjecture on Hecke operators. Then we formulate  the Arithmetic Fundamental Lemma conjecture for the spherical Hecke algebra. We also formulate a conjecture on the abundance of spherical Hecke functions with identically vanishing first derivative of orbital integrals. We prove these conjectures for the case $\RU(1)\times\RU(2)$.

\end{abstract}

\maketitle{}
\tableofcontents{}
\section{Introduction}
 
 In the relative trace formula approach  of Jacquet and Rallis to the Gan-Gross-Prasad conjecture, the Jacquet-Rallis fundamental lemma (FL) conjecture  plays a key role \cite{JL}. It states an identity of the following form.   Let $p$ be an odd prime number. Let $F_0$ be a finite extension of $\BQ_p$ and let $F/F_0$ be an unramified quadratic extension. Let $W_0$ be a split $F/F_0$-hermitian space of dimension $n+1$ and let $W_0^\flat$ be the perp-space of a vector $u_0\in W_0$ of unit length. Then the following identity holds for all \emph{matching regular semi-simple} elements $\gamma\in \GL_n(F)\times\GL_{n+1}(F)$ 
and 
 $g\in \U(W_0^\flat)(F_0)\times\U(W_0)(F_0)$,
 \begin{equation}
 \Orb(g, {\bf 1}_{K^\flat\times K})=\omega(\gamma)\Orb(\gamma,{\bf 1}_{K'^\flat\times K'}) .
 \end{equation}
Here on the RHS, there appears the weighted orbital integral of the characteristic function of the natural hyperspecial compact subgroup $K'^\flat\times K'$ of $\GL_n(F)\times\GL_{n+1}(F)$; on the LHS there appears the orbital integral of the characteristic function of the natural hyperspecial compact subgroup $K^\flat\times K$ of $\U(W_0^\flat)(F_0)\times\U(W_0)(F_0)$. 
The first factor on the RHS is the natural transfer factor, cf. \cite{RSZ2}; both sides  only depend on the orbits of $\gamma$, resp. $g$, under natural group actions.

The FL conjecture was proved for $F$ with large residue characteristic by Yun and Gordon \cite{Yun}, and is now proved completely:  the proof by R. Beuzart-Plessis \cite{BP} is  local; the proof in \cite{Zha21} (for $p\geq n+1$) is global.  

In fact, an identity of this form is true for the whole spherical Hecke algebra, and is due to S.~Leslie \cite{Les}. Let $\varphi'\in \CH_{K^{\prime \flat}\times K'}$ be an arbitrary element in the spherical Hecke algebra for $\GL_n(F)\times\GL_{n+1}(F)$. Then the following identity holds for all matching regular semi-simple elements $\gamma\in \GL_n(F)\times\GL_{n+1}(F)$ 
and 
 $g\in \U(W_0^\flat)(F_0)\times\U(W_0)(F_0)$,
 \begin{equation}
 \Orb(g,\varphi)=\omega(\gamma)\Orb(\gamma,\varphi') .
 \end{equation}
 Here on the LHS appears the orbital integral of the image $\varphi$ of $\varphi'$ under the \emph{base change homomorphism}  from the spherical Hecke algebra of $ \GL_n(F)\times\GL_{n+1}(F)$ to the spherical Hecke algebra of $ \U(W_0^\flat)(F_0)\times\U(W_0)(F_0)$.  The second factor on the RHS is  the weighted orbital integral of $\varphi'$. The method of proof of \cite{Les} is ultimately global.

The third author proposed a relative trace formula approach to the \emph{arithmetic} Gan-Gross-Prasad conjecture. In this context, he formulated the arithmetic fundamental lemma (AFL) conjecture \cite{Zha12}. The AFL  relates the special value of the derivative of an orbital integral to an arithmetic intersection number on a Rapoport-Zink formal moduli space (RZ space) of $p$-divisible groups attached to a unitary group.  The AFL conjecture is an identity of the following form. Let $W_1$ be a non-split $F/F_0$-hermitian space of dimension $n+1$ and let $W_1^\flat$ be the perp-space of a vector $u_1\in W_1$ of unit length. Then the following identity holds for all matching regular semi-simple elements $\gamma\in \GL_n(F)\times\GL_{n+1}(F)$ 
and 
 $g\in \U(W_1^\flat)(F_0)\times\U(W_1)(F_0)$,
\begin{equation}\label{IntroAFL}
2\langle g\Delta, \Delta\rangle_{\CN_{n, n+1}}\cdot\log q=- \omega(\gamma)\del(\gamma, {\mathbf 1}) .
\end{equation}
Here the second factor on the RHS is the special value of the derivative of the weighted orbital integral of the unit element in the spherical Hecke algebra $\CH_{K^{\prime \flat}\times K'}$. On the LHS appears the intersection number of the diagonal cycle $\Delta$ of the product RZ-space $\CN_{n, n+1}=\CN_n\times\CN_{n+1}$ with its translate under the automorphism of $\CN_{n, n+1}$ induced by $g$. Here, for any $n$,  $\CN_n$ is the moduli space of framed \emph{basic} principally polarized $p$-divisible groups with action of $O_F$ of signature $(1, n-1)$. 

The AFL conjecture is now known to hold for any odd prime $p$, cf. W. Zhang \cite{Zha21}, Mihatsch-Zhang \cite{MZ}, Z. Zhang \cite{ZZha}.  These proofs are global in nature. Local proofs of the AFL are known for $n=1,2$ (W. Zhang \cite{Zha12}),  and for minuscule elements (He-Li-Zhu  \cite{HLZ}).

The aim of the present paper is to propose a variant of the AFL conjecture in the spirit of Leslie's result on the FL, where the unit element in the spherical Hecke algebra is replaced by an arbitrary element $\varphi'\in \CH_{K^{\prime \flat}\times K'}$. The proposed formula takes the following form,
\begin{equation}\label{IntrogenAFL}
2\langle g\Delta, \BT_{\varphi}(\Delta)\rangle_{\CN_{n, n+1}}\cdot\log q=- \omega(\gamma)\del(\gamma, \varphi') .
\end{equation}
The new feature compared to the AFL conjecture \eqref{IntroAFL} for the unit element is the appearance of the Hecke operator $\BT_{\varphi}$ on the LHS, and the definition of such Hecke operators is one of the main issues of the present paper, see below.

The AFL conjecture comes, as usual, in a homogeneous version (as stated above) and an inhomogeneous version. However, in contrast to the case of the unit element, these two versions are not equivalent: the homogeneous version implies the inhomogeneous version but not conversely. It is conceivable that the inhomogeneous version is easier to prove in some cases. 

As evidence for this conjecture, we prove it in the case $n=1$ (in this case, the homogeneous version and the inhomogeneous version are easily seen to  be equivalent). \begin{theorem}\label{Intro:main=1}
The AFL formula \eqref{IntrogenAFL}  holds for $n=1$. 
\end{theorem}
 The proof is local, 
by explicit calculation of both sides of the formula and resembles the proof of the AFL in cases of low rank in \cite{Zha12}. On the geometric side we exploit the fact that, in the particular case $n=1$, the Hecke operators are induced by explicit geometric correspondences which are finite and flat. Another ingredient is the theory of quasi-canonical divisors on $\CN_2$ in the sense of  \cite{Kudla2011}. On the analytic side, we also give a purely local proof of the FL for the whole Hecke algebra (Leslie's theorem) in this case.

Our definition of Hecke operators (in K-theory)  is based on the fact that there is a presentation of the spherical Hecke algebra of the unitary group as a polynomial algebra. The basis elements of this presentation  can be chosen to be decomposed into a product (i.e., convolution) of {\it intertwining Hecke functions}. Here an intertwining Hecke function in the Iwahori Hecke algebra for a fixed Iwahori subgroup  is a function of the form ${\bf 1}_{KK'} $ for  parahoric subgroups $K, K'$ stabilizing a facet in the  alcove in  the Bruhat--Tits building corresponding to the Iwahori subgroup\footnote{The  terminology ``intertwining Hecke function" is borrowed from \cite[Definition B.2.3]{LTXZZ}, where a special case is considered.}.
For these elements it is possible  to define integral models of Hecke correspondences. Indeed,  we can naturally define a geometric correspondence between two RZ spaces  for parahorics $K,K'$ as above, a diagram of RZ spaces at  parahoric levels $K$, $K'$, $K\cap K'$, 
\begin{equation}\label{Hkdiag}
\begin{aligned}
\xymatrix{&\CN_{K\cap K'}\ar[ld]_{\pi_1} \ar[rd]^{\pi_2}  & \\ \CN_K  & &\CN_{K'}  .}
\end{aligned}
\end{equation}
 In fact, for our purposes, it suffices to consider intertwining Hecke correspondences of the form  ${\bf 1}_{K K'}$ and  ${\bf 1}_{K' K}$, where $K'$ is a maximal parahoric and $K$ is the  hyperspecial vertex in the fixed alcove defining the spherical Hecke algebra. Hecke operators for   general elements in the spherical Hecke algebra, in the sense of maps on K-groups, are then defined in three steps. The basis elements, also called \emph{atomic elements}, can be written in the form $\phi_{K'}=\vol(K')^{-1}{\bf 1}_{K K'}*{\bf 1}_{K' K}$, where $K'$ is a maximal parahoric subgroup corresponding to a vertex in the fixed alcove, and this defines the corresponding Hecke operator. For a monomial in the atomic  elements, the corresponding Hecke operator is defined as a product of atomic  Hecke operators. The general case is obtained by linear combinations.  However, at this point arises a highly non-trivial problem: the definition of monomial Hecke operators presupposes that the Hecke operators corresponding to different $\phi_{K'}$ commute. We conjecture that this is indeed true but at the moment our formulation of the AFL formula \eqref{IntrogenAFL} is contingent on the solution of this conjecture. More precisely, without this conjecture,  the formulation of the AFL conjecture becomes somewhat awkward (cf. Remark \ref{explain}), unless $\varphi$ is a power of an atomic element (but even this instance of the AFL formula may be interesting to prove).
 
 The idea of defining Hecke operators as linear combinations of products of certain distinguished Hecke operators also appears in the work of Li--Mihatsch on the linear ATC \cite{LM}. In their case, the projection maps $\pi_1$ and $\pi_2$ are finite and flat, and the same is true for  compositions of distinguished Hecke correspondences. This implies that their Hecke operators are induced by explicit geometric Hecke correspondences. This also allows them to pass to the generic fiber to prove the necessary commutativity statement in their context. 
 
 In our case the projection maps $\pi_1$ and $\pi_2$  are usually not flat and the composition of such correspondences, in the sense of maps on K-groups, is not induced by the composition of geometric correspondences.  This discrepancy between compositions of geometric correspondences and K-group correspondences would disappear if instead of usual (\emph{classical}) formal schemes we had used \emph{derived} formal schemes in the definition of geometric correspondences. In this sense, our definition of Hecke operators  is a ``shadow'' of a more sophisticated definition (which remains to be developed). By avoiding  derived schemes, we forgo the possibility of defining our Hecke operators in terms of geometric Hecke correspondences\footnote{It may be possible, at least as far as defining and calculating intersection multiplicities is concerned, to replace  derived formal schemes by their underlying classical formal scheme, equipped with a suitable element in the derived category of coherent sheaves.}. However, it is unclear to us whether such a more sophisticated definition can be helpful in resolving the commutativity conjecture mentioned above. Relatedly, it seems that the more sophisticated definition of Hecke correspondences transfers to the global context of integral models of Shimura varieties for $\GU(1, n-1)$ but  the relation to the classical Hecke correspondences in the generic fiber is unclear.

 Let us compare our construction of Hecke operators  with variants in the literature; indeed, the construction of integral Hecke correspondences and their induced Hecke operators on cohomology, resp. cycle groups, resp. K-groups is a well-known problem in various contexts.  An example, in the context of integral models of Shimura varieties, occurs in the proof of the Eichler-Shimura congruence relation, comp. Faltings-Chai \cite{FC} and B\"ultel-Wedhorn \cite{BW}, Koskivirta \cite{Ko}, Lee \cite{Lee}, Wedhorn \cite{Wed}. Specifically, in the case of the Siegel moduli space with hyperspecial level at $p$, one considers simply the space of all isogenies of $p$-power degree and then isolates inside it a subspace that can be analyzed for the purpose at hand (note that the space of all isogenies of $p$-power degree is an unwieldy object that is hard to control). In the function field context, the problem of defining integral Hecke operators (in cohomology) is addressed by Lafforgue \cite{Laf} by using his \emph{excursion operators} (see loc.~cit., Prop. 6.2); in this context, he also solves a commutativity problem (see loc.~cit., Lem. 10.1, equ. (10.4)).
Let us also mention the recent paper by Fakhrudin-Pilloni \cite{FP}, in which they aim to define Hecke operators for automorphic vector bundles on $p$-integral models of Shimura varieties with hyperspecial level at $p$. Translated to our language of RZ spaces, they consider correspondences given by diagrams \eqref{Hkdiag}, where $K$ and $K'$ are hyperspecial and conjugate under an auxiliary group (in their case, a group of unitary or symplectic similitudes). It should be pointed out that such diagrams exist only rarely:  in the case of the symplectic group, there is precisely one such diagram ($K$ is the stabilizer of a selfdual lattice and $K'$ is the stabilizer of a lattice selfdual up to a scalar),  and similarly in the case of unitary groups considered here  in the even rank case when  $K$ is the stabilizer of a selfdual lattice and $K'$ is the stabilizer of a lattice selfdual up to a scalar; in the case of the general linear group, all pairs $K, K'$ of hyperspecial subgroups corresponding to vertices in a fixed alcove give such diagrams. On the other hand, in \cite[\S 7]{Pil} Pilloni defines more general automorphic vector bundle Hecke operators for ${\rm GSp}_4$ in a way  somewhat similar to ours, via intertwining Hecke operators. It is interesting to note that in the context of \cite{FP}, there is also a commutativity conjecture of Hecke operators \cite[Rem. 7.6]{FP}; however, there seems to be no direct relation to our conjecture above (but maybe a solution to one of the problems can give indications for a solution to the other problem).
We also note that a function field analog has been considered by Yun and the third author (cf. \cite[Prop. 5.10]{YZ} and \cite[Prop. 3.14]{YZ2}), where they consider the moduli space of $\GL_2$-shtukas (with an arbitrary number of legs) and construct  Hecke correspondences for a natural basis (as a vector space) of the spherical Hecke algebra. They show a commutativity statement using crucially an equidimensionality result (cf. \cite[Lem. 5.9]{YZ} and  \cite[Lem. 3.13]{YZ2}), which is in turn proved via constructions closely related to the Geometric Satake isomorphism. Another attempt at defining integral Hecke correspondences occurs for RZ spaces in \cite[Chap. 4]{RZ96}. That definition suffers from several drawbacks, the most serious being that the projection morphisms may not be proper and not surjective, cf. \cite[remark after Prop. 4.44]{RZ96}.

Why is it of interest to extend the AFL conjecture from the unit element to all elements in the spherical Hecke algebra? The reason that  
in the proof of the global Gan--Gross--Prasad conjecture (e.g., \cite{Z14,BPLZZ}) one only considers the FL for the unit element is the  density theorem of Ramakrishnan \cite{Ram}. It allows one to avoid the Jacquet--Rallis fundamental lemma for the full spherical Hecke algebra at inert places. However, such a density result is not available for the orthogonal group, in which case we need necessarily to consider the full Hecke algebra.  It should be pointed out, however, that at present we do not have a formulation of an FL conjecture or an AFL conjecture in the case of the orthogonal group. Another motivation comes from the consideration of the $p$-adic height pairing of  arithmetic diagonal cycles, as in  on-going work of Disegni and the third author \cite{DZ}.  Here it is necessary to consider all Hecke correspondences at $p$-adic places. This is one of  the reasons, why in \cite{DZ} it is assumed that all  $p$-adic places are split in the quadratic extension of global fields $F/F_0$.  When there are  inert $p$-adic places, it will be necessary to consider  Hecke correspondences at inert places and  the situation of the present paper becomes relevant. One may even need to consider the  more complicated case of the Iwahori level Hecke algebra.  In fact, in  Disegni's work  on the $p$-adic Gross--Zagier formula for Shimura curves in the inert case \cite{Di}, a crucial ingredient are  Hecke correspondences for  arbitrarily deep level. 

One spin-off of the consideration of the AFL conjecture for the spherical Hecke algebra is that it naturally leads to the following  question, also partly motivated by Disegni's work. Namely, one may ask whether a function in  the spherical Hecke algebra is determined by its  first derivatives of orbital integrals over regular semi-simple elements.  To put this into context, it should be pointed out that a function in the  spherical Hecke algebra is determined by its  orbital integrals over regular semi-simple elements, cf. Proposition \ref{B-P}. Experimental evidence points to the fact that these two questions have quite distinct answers. Indeed, we conjecture that there is an  abundance of  functions with vanishing first derivatives of orbital integrals, in the following precise form. 

Let $G'_{\rs, W_1}$ denote the open subset of $\GL_n(F)\times\GL_{n+1}(F)$ consisting of regular semisimple elements matching with elements in the non-quasi-split unitary group $\U(W_1^\flat)(F_0)\times\U(W_1)(F_0)$ .
\begin{conjecture}\label{Intro-conj:avoid}
The map
$$
\del: \CH_{K^{\prime \flat}\times K'}\to C^\infty(G'_{\rs, W_1})
$$
has a large kernel, in the sense that the kernel generates the whole ring $\CH_{K^{\prime \flat}\times K'}$ as an ideal (note that this kernel is only a vector subspace rather than an ideal). Similarly,  the map defined by the intersection numbers, $\Int:  \CH_{K^\flat\times K}\to C^\infty(G'_{\rs,W_1})$, has a large kernel.
\end{conjecture}
The conjecture is somewhat speculative and we give several weaker variants of it. We confirm this conjecture  in the case $n=1$.

\begin{theorem}\label{Intro-thm: n=1}
Conjecture \ref{Intro-conj:avoid} holds when $n=1$.
\end{theorem}

However, even without spin-offs, the arithmetic fundamental lemma for the entire spherical Hecke algebra is  an interesting problem  of its own, which may turn out to be quite difficult. Its solution might yield additional insight into the nature of Rapoport-Zink spaces and their special cycles. It might also be a good testing ground for applying derived algebraic geometry in an unequal characteristic situation, after its success in the equal characteristic counterpart, comp. \cite{FYZ}.  It would also be interesting to consider Hecke correspondences  for other RZ spaces.

Since the  arithmetic fundamental lemma for the entire spherical Hecke algebra seems so difficult, it may be instructive to prove it in special cases. We already mentioned its inhomogeneous version. Another possible simplification may occur for Hecke functions of the form $\varphi'={\bf 1}_{K^{\prime \flat}}\otimes f'$, where $f'\in\CH_{K'}$. Yet another simplification may occur for atomic elements. The statement for Hecke functions $\varphi'$ with base change of the form  $\varphi=f\otimes {\bf 1}_{K}$, resp.  $\varphi={\bf 1}_{K^{\flat}}\otimes f'$,  where $f\in \CH_K$, resp. $f'\in\CH_{K^\flat}$ is an atomic element, are of key importance in our paper \cite{LRZ}. 

 We note that our procedure is based on   integral models of Hecke correspondences for certain elements that are not contained in  the spherical Hecke algebra (the  function ${\bf 1}_{KK'}$ is usually not spherical). Now ${\bf 1}_{KK'}$ is contained in the Iwahori   Hecke algebra (corresponding to a  fixed chamber containing the facets corresponding to $K$ and $K'$),  and it would be interesting to see how large a subalgebra they generate, in order to define the LHS of \eqref{IntrogenAFL} for a wider class of functions $\varphi$. On the other hand, at this moment we do not know natural candidates of smooth transfers (in the sense of Jacquet--Rallis) for functions not in the spherical Hecke algebra, so we have no immediate use for Arithmetic Transfer conjectures for all integral Hecke correspondences that we construct in this way.

In view of the global nature of the proof of Leslie's theorem  and of the various FL statements for full Hecke algebras in the Langlands program, it is natural to speculate that a proof of our AFL conjecture would necessarily require a global input. It seems that the most promising approach is to study the global $p$-adic height pairing of Nekov\'a\v{r}, which satisfies a ``modularity" condition, in the sense that the action of the Hecke algebra on this height pairing factors through an action on automorphic forms. One may leverage this modularity to deduce a version of the (global) relative trace formula identities for the {\em full Hecke algebra}, from the a priori weaker result for a {\em partial Hecke algebra} (i.e., locally the unit element at inert place and arbitrary at split places). One may then hope to deduce from such global identities the AFL identity \eqref{IntrogenAFL}.  The aforementioned work of Disegni and the third author on $p$-adic height pairings  \cite{DZ} involves the  projection to the \emph{ordinary part}; this is an obstacle to pushing through this  proof strategy. 

The layout of the paper is as follows. After a notation section,  we recall in \S3 the set-up and review the formulation of the Jacquet--Rallis transfer and Fundamental Lemma for the full spherical Hecke algebra. In \S4 we define  atomic Hecke functions and  exhibit the spherical Hecke algebra  
as the polynomial algebra of the atomic elements. In \S5 we define various Rapoport--Zink spaces with certain parahoric levels, and use them to define Hecke correspondences. We then formulate the commutativity conjecture.  In \S6 we state the AFL conjecture for the spherical Hecke algebra, and in \S7 we prove it in the case $\RU(1)\times\RU(2)$.  In \S8 we formulate a conjecture on the abundance of spherical Hecke functions with identically vanishing first derivative of orbital integrals. In \S9 we collect a few facts on the general theory of correspondences.

We thank Tony Feng, Benjamin Howard and Andreas Mihatsch for their help. We thank SLMath for its hospitality to all three of us during the Spring 2023 semester on ``Algebraic cycles, L-values, and Euler systems"  when part of this work was done. We also thank the referee for his excellent work.

\section{Notations}

Let $p>2$ be a prime.  Let $F_0$ be a finite extension of $\mathbb{Q}_p$, with ring of integers $O_{F_0}$, residue field $k=\mathbb{F}_q$ of size $q$, and uniformizer $\varpi$. Let $F$ be the unramified quadratic extension of $F_0$, with ring of integers $O_{F}$ and residue field $k_F$.  Let $\sigma$ be the nontrivial Galois automorphism of $F/F_0$. Fix $\delta\in O_F^\times$ such that $\sigma(\delta)=-\delta$. Let $\val:F\to \BZ\cup\{\infty\}$ be the valuation on $F$. Let $|\cdot|_F:F\rightarrow \mathbb{R}_{\ge0}$ (resp. $|\cdot|: F_0\rightarrow \mathbb{R}_{\ge0}$) be the normalized absolute value on $F$ (resp. $F_0$). Let $\eta=\eta_{F/F_0}: F_0^\times \rightarrow\{\pm1\}$ be the quadratic character associated to $F/F_0$. We let $\wt\eta: F^\times  \rightarrow\{\pm1\}$ be the unique unramified quadratic character extending $\eta$. Let $\Fb$ be the completion of the maximal unramified extension of $F$, and $\OFb$ its ring of integers, and $\bar k$ its residue field. 

A $F/F_0$-hermitian space $W$ is called \emph{split} if there exists a Witt basis; otherwise $W$ is called \emph{non-split}. Equivalently, $W$ is split if 
$\eta_{F/F_0}((-1)^{\frac{n(n-1)}{2}}\det(W))=1$.  For a vector $u\in W$ we denote by $\pair{u}$ (resp. $\pair{u}_F$) the $O_F$-submodule (resp. the $F$-submodule) generated by $u$.

For an algebraic variety $X$ over $F_0$, we use the notation $C_0^\infty(X)$ for $C_0^\infty(X(F_0))$.

\section{FL for the full spherical Hecke algebra}\label{s:FL}

In this section we review the formulation of the Jacquet--Rallis transfer and Fundamental Lemma for the full spherical Hecke algebra.
\subsection{Groups} We recall the group-theoretic setup of \cite[\S2]{RSZ1} in both homogeneous and inhomogeneous settings. Let $n\geq 1$. In the homogeneous setting, set
\begin{equation}
G':=\Res_{F/F_0}(\GL_{n}\times\GL_{n+1}),
\end{equation}
a reductive algebraic group over $F_0$. Let $W$ be  a $F/F_0$-hermitian space of dimension $n+1$.  Fix $u\in W$  a non-isotropic vector (the {\it special vector}), and let $W^\flat=\langle u\rangle^\perp$. Set
\begin{equation}
G_W=\U(W^\flat)\times \U(W),
\end{equation}
a reductive algebraic group over $F_0$. We have the notion of a {\it regular semi-simple element}, for $\gamma\in G'(F_0)$ and for $g\in G_W(F_0)$. The  notions of regular semi-simple elements are with respect to the action of the reductive algebraic group over $F_0$, 
$$H'_{1, 2}=H_1'\times H_2' :=\Res_{F/F_0} (\GL_{n})\times ( \GL_{n}\times\GL_{n+1})$$
 on $G'$, resp., of $H_{1, 2}=\U(W^\flat)\times \U(W^\flat)$ on $G_W$. The sets of regular semi-simple elements are denoted by $G'(F_0)_\rs$ and $G_W(F_0)_\rs$ respectively.  We choose a basis of $W$ by first choosing a basis of $W^\flat$ and then adding the special vector as the last basis vector. This then  gives an identification of $G_W(F)$ with $G'(F_0)$, and defines the notion of {\it matching} $\gamma\leftrightarrow g$ between regular semi-simple elements of $G_W(F_0)$ and $G'(F_0)$, cf. \cite[\S2]{RSZ1}.

In the inhomogeneous setting, recall the symmetric space 
\begin{equation}\label{defsy}
S = S_{n+1} := \{ g \in \Res_{F/F_0}\GL_{n+1} \mid g \ov g = 1_{n+1}\}
\end{equation}
and the map $r: \Res_{F/F_0}\GL_{n+1}\to S$ given by $g\mapsto \gamma=g\ov g^{-1}$, which induces an isomorphism $$(\Res_{F/F_0}\GL_{n+1})/\GL_{n+1}\simeq S.$$ We have the notion of a \emph{regular semi-simple element}, for $\gamma\in S(F_0)$ and for $g\in \U(W)(F_0)$ and, after the choice of a basis of $W$ as above,  the notion of \emph{matching} $\gamma\leftrightarrow g$. The  notions of regular semi-simple elements are with respect to the conjugation actions of $H':=\GL_{n}$ on $S$, resp., of $H:=\U(W^\flat)$ on $\U(W)$. The sets of regular semi-simple elements are denoted by $S(F_0)_\rs$ and $\U(W)(F_0)_\rs$ respectively.

\subsection{Orbital integrals}

  We recall the orbital integrals in both homogeneous and inhomogeneous settings, following \cite[\S5]{RSZ1}. In the homogeneous setting, for $\gamma\in G'(F_0)_\rs$, a function $\fp'\in C^\infty_0(G')$ and a complex parameter $s\in \mathbb{C}$, we define 
\begin{equation}\label{def Orb s}
   \Orb(\gamma, \fp', s) := \int_{H_{1,2}'(F_0)} \fp'(h_1^{-1}\gamma h_2) \lv\det h_1\rv_F^s \eta(\det h_2)\, \rd h_1\, \rd h_2,
\end{equation}
where we use the normalized Haar measures on $H_1'(F_0)$ and $H_2'(F_0)$ (such that the obvious hyperspecial maximal compact open groups have measure one) and the product Haar measure on $H_{1,2}'(F_0) = H_1'(F_0) \times H_2'(F_0)$. 
We further define the value and derivative at $s=0$,
\begin{equation}
   \Orb(\gamma, \fp') := \Orb(\gamma, \fp', 0)
	\quad\text{and}\quad
	\del(\gamma, \fp') := \frac{\rd}{\rd s} \Big|_{s=0} \Orb({\gamma},  \fp',s) . 
\end{equation}
The integral defining $\Orb(\gamma,\fp',s)$ is absolutely convergent, and (upon introducing a suitable transfer factor) $\Orb(\gamma, \fp')$ and $\del(\gamma, \fp')$ depend only on the orbit of $\gamma$.

Now we turn to the inhomogeneous setting.  For  $\gamma\in S(F_0)_\rs$,  a function $\phi\in C_c^\infty(S)$, and a complex parameter $s\in \BC$, we introduce the \emph{weighted orbital integral}
\begin{align}\label{eqn def inhom}
   \Orb(\gamma,\phi', s) := \int_{H'(F_0)}\phi'(h^{-1}\gamma h)\lvert \det h \rvert^s \eta(\det h) \, \rd h ,
\end{align}
as well as the value and derivative at $s=0$,
\begin{equation*}
   \Orb(\gamma,\phi') := \Orb(\gamma,\phi', 0)
   \quad\text{and}\quad
   \del(\gamma,\phi') : = \frac \rd{\rd s} \Big|_{s=0} \Orb(\gamma, \phi',s) . 
\end{equation*}
  As in the homogeneous setting, the integral defining $\Orb(\gamma,\phi', s)$ is absolutely convergent, and (upon introducing a suitable transfer factor)  $\Orb(\gamma, \phi')$ and $\del(\gamma, \phi')$ depend only on the orbit of $\gamma$.
  
There are also corresponding orbital integrals on the unitary side. In the homogeneous setting, for $g\in G_{W}(F_0)_\rs$ and a function $\fp\in C^\infty_0(G_{W})$, we define
\begin{equation}\label{def Orb unhom}
   \Orb(g, \fp) := \int_{H_{1,2}(F_0)} \fp(h_1^{-1}\gamma h_2) \, \rd h_1\, \rd h_2.
\end{equation}
When $W^\flat$ is split, we  normalized the Haar measure on $\U(W^\flat)$ such that the hyperspecial maximal compact subgroups get measure one, and take the product Haar measure on $H_{1,2}(F_0)$.  In the inhomogeneous setting, we set for $g\in \U(F_0)_\rs$ and  a function $\phi\in C_c^\infty(U(W))$, 
  \begin{equation}\label{def Orb uninhom}
   \Orb(g, \phi) := \int_{\U(F_0)} \phi(h^{-1} g h) \, \rd h .
\end{equation}

\subsection{Matching and transfer} Let $W_0$, $W_1$ be representatives of the two isomorphism classes of $F/F_0$-hermitian spaces of dimension $n+1$. We assume $W_0$ to be split. Take the special vectors  $u_0\in W_0$ and $u_1\in W_1$ to have the same norm (not necessarily a unit). We also choose bases of $W_0$ and $W_1$ as above. Then matching defines 
 bijections of regular semisimple orbits 
\begin{equation} \label{eq:orb hom} \xymatrix{ \bigl[G_{W_0} (F_0)\bigr]_\rs 
 \bigsqcup\,\bigl[ G_{W_1} (F_0)\bigr]_\rs  \ar[r]^-\sim& \bigr[G'(F_0)\bigr]_\rs}
\end{equation}
in the homogeneous setting,
and
\begin{equation} \label{eq:orb inhom}
\xymatrix{ \bigl[ \U(W_0) (F_0)\bigr]_\rs 
  \bigsqcup\,\bigl[ \U(W_1) (F_0)\bigr]_\rs  \ar[r]^-\sim& \bigr[S(F_0)\bigr]_\rs}
\end{equation}
in the inhomogeneous setting, cf. \cite[\S\S  2.1, 2.2]{RSZ1}.

Then associated to a transfer factor $\omega_{G'}: G'(F_0)_\rs\rightarrow \mathbb{C}^\times$ we have the notion of \emph{transfer} between functions $\fp'\in C_c^\infty(G')$ and pairs of functions $(f_0,f_1)\in C_c^\infty(G_{W_0})\times C_c^\infty(G_{W_1})$ (\cite[Definition 2.2]{RSZ2}). We will always use the transfer factor given by \cite[(5.2)]{RSZ2} (extrapolated in the obvious way from odd $n$ to even $n$). Similarly,  for a transfer factor $\omega_S: S(F_0)_\rs\rightarrow \mathbb{C}^\times$ we have the notion of \emph{transfer}  between functions $\phi'\in C_c^\infty(S)$ and pairs of functions $(f_0,f_1)\in C_c^\infty(\U(W_0))\times C_c^\infty(\U(W_1))$ (\cite[Definition 2.4]{RSZ2}) We will always use the transfer factor given by \cite[(5.5)]{RSZ2} (again extrapolated to all $n$).

\subsection{Satake isomorphism and base change}\label{ss:Hk}

Let $\K=K'_n=\GL_n(O_{F})$. We denote by $\Hk=\BQ[K'_n\backslash \GL_n(F)/K'_n]$  the Hecke algebra with coefficients in $\BQ$ of $\GL_n$. Haar measures are chosen such that maximal compact subgroups have measure one.

 We first recall the Satake isomorphism for $\GL_n$ over $F$. Let $T\subseteq\GL_n$ be the diagonal torus. The cocharacter group $X_*(T)$ is a free abelian group generated by $\{\mu_1,\ldots, \mu_n\}$, where $\mu_i$ is the injection in the $i$-th factor. For $\mu\in X_*(T)$, denote by $[\mu]$ the corresponding element in the group algebra $\mathbb{C}[X_*(T)]$. For $1\le i\le n$, define 
\begin{align}\label{mu i}
x_i:=[\mu_i]\in \mathbb{C}[X_*(T)].
\end{align}
We identify $\mathbb{C}[X_*(T)]$ with  $\mathbb{C}[T(F)/T(O_F)]$ by sending $\lambda\in X_*(T)$ to the monomial diagonal matrix $\varpi^\lambda$. Let $\sig_i$ be the degree $i$ elementary symmetric polynomial in $\{x_1,\ldots,x_n\}$, which corresponds to the sum of the elements in the $S_n$-orbit of $\varpi^{(1^i, 0^{n-i})}T(O_F)$. Then the Satake transform gives an isomorphism of algebras 
\begin{equation}\label{Sat GL}
\Sat: \Hk\otimes_\BQ \BC\isoarrow  \mathbb{C}[\sig_1,\ldots,\sig_{n-1}, \sig_n^{\pm}]=\mathbb{C}[T(F)/T(O_F)]^{S_n}.
\end{equation}
 It sends the minuscule function $\mathbf{1}_{\K\varpi^{(1^i, 0^{n-i})}\K}$ to $q_F^{i(n-i)/2}\sig_i=q^{i(n-i)}\sigma_i$.  Note that since the modulus function $\delta_B^{\frac{1}{2}}$ takes values in $\BQ$, the homomorphism \eqref{Sat GL} is defined over $\BQ$.

Next we recall the Satake isomorphism for the unramified unitary group. Let now $W_0$ be a split $F/F_0$-hermitian space of dimension $n$. Let $m=\lfloor n/2\rfloor$. We choose a basis of $W_0$ such that the hermitian form is given by the antidiagonal unit matrix.   Let $\Xi\subset W_0$ be the standard lattice, which is self-dual, and let  $K\subset \U(W_0)(F_0)$ be its stabilizer. Let $\hk=\BQ[K\backslash \U(W_0)(F_0)/K]$  be the Hecke algebra with coefficients in $\BQ$ of $\U(W_0)$.   We recall the Satake isomorphism for $\U(W_0)$. Let $A$ be the maximal split diagonal torus in $\U(W_0)$. With the chosen basis,  we can identify $A$ with the torus consisting of diagonal elements of the form $\diag(x_1,\cdots,x_m,1_{n-2m}, x_{m}^{-1},\cdots, x_1^{-1})$ in $\U(W_0)$. For $1\le s\le m$, let $\nu_s=\mu_s-\mu_{n+1-s}\in X_*(A)$ be the cocharacter of $A$ sending $x$ to $\diag(1,\cdots,x_s=x,\cdots,1,1_{n-2m}, 1, \cdots, x_{n+1-s}=x^{-1},\cdots, 1)$. For $1\le s\le m$ we  define
$$
y_s:=[\nu_{s}]+[-\nu_{s}]\in \mathbb{C}[X_*(A)].
$$
 Let $\si_s$ be the degree $s$ elementary symmetric polynomial in $\{y_1,\ldots,y_m\}$.  Then the Satake transform gives an isomorphism of algebras 
\begin{equation}\label{Sat U}
\Sat: \hk\otimes_\BQ \BC\isoarrow \mathbb{C}[\si_1,\ldots, \si_m]=\mathbb{C}[A(F_0)/A(O_{F_0})]^{W_n},
\end{equation}
 where $W_n\simeq (\mathbb{Z}/2 \mathbb{Z})^m\rtimes S_m$ is the Weyl group of $A(F_0)$ in $\U(W_0)(F_0)$. In particular, $\hk$ is a polynomial algebra. Analogous to  the case of $\GL_n$ (over $F$), the modulus  function $\delta_B^{\frac{1}{2}}$ takes values in $\BQ$, hence the homomorphism \eqref{Sat U} is defined over $\BQ$. 

We have an algebra homomorphism, called  the base change homomorphism,
\begin{equation}
\Bc:\Hk\rightarrow \hk,\quad \fp'\mapsto \Bc(\fp') .
\end{equation}
  The homomorphism $\Bc$ is  characterized by the identity 
  
  $$
  {\rm trace}\, \Pi(\fp')=  {\rm trace}\, \pi( \Bc(\fp')), 
  $$where $\Pi$ denotes the base change of $\pi$, for all unramified representations $\pi$ of $ \U(W_0)$.

  An alternative description is in terms of a morphism of the Langlands dual tori $\wh A\to \wh T$. More precisely, we may identify 
  $\wh T$ with $ (\BC^\times)^n$, sending the character  $\mu_i\in X^\ast(\wh T)=X_{\ast}(T)$ of $\wh T$ to the $i$-th coordinate $\alpha_i$ on $ (\BC^\times)^n$. We may identify $\wh A$ as the subtorus   $ (\BC^n)^{\rm unit}$ of $\wh T= (\BC^\times)^n$, where
   $ (\BC^n)^{\rm unit}$ denotes the space of  unitary parameters in $\BC^n$,
\begin{equation}
\begin{aligned}
 (\BC^n)^{\rm unit}=\{ \alpha=(\alpha_1,\ldots,\alpha_n)\in \mathbb{C}^n\mid \alpha_i\alpha_{n+1-i}=1\,\,  \forall  i,\\\text{and $\alpha_{m+1}=1$ if $n=2m+1$}\},
\end{aligned}
\end{equation}
such that the   character  $\nu_i\in X^\ast(\wh A)=X_{\ast}(A)$ of $\wh A$ corresponds to the  $i$-th coordinate $\alpha_i$.  In particular, the element $y_s=[\nu_{s}]+[-\nu_{s}]\in \mathbb{C}[X_*(A)]=\mathbb{C}[X^*(\wh A)]$ corresponds to $\alpha_s+\alpha_{s}^{-1}$ as regular functions on $\wh A$.
The morphism $\wh A\to \wh T$ is then the natural inclusion \[
{\rm bc}\colon (\BC^n)^{\rm unit}\hookrightarrow (\BC^\times)^n,
\]
and the base change homomorphism is characterized by
$$
\Sat(\fp')({\rm bc}(\alpha))=\Sat(\Bc(\fp'))(\alpha) .
$$
We denote by the same symbol  the induced homomorphism,
$$\Bc\colon \BQ[\sigma_1,\sigma_2,\ldots,\sigma_{n-1}, \sigma_n^{\pm}]\to \BQ[\si_1,\ldots, \si_m].
$$ We obtain a commutative diagram 
\begin{equation}
\begin{aligned}
  \xymatrix{\CH_{K'} \ar[r]^-{\rm Sat} \ar[d]_{\Bc} \ar@{}[rd]  & \BQ[\sigma_1,\sigma_2,\ldots,\sigma_{n-1}, \sigma_n^{\pm}] \ar[d]^{\Bc} \\ 
  \CH_K \ar[r]_-{\rm Sat} & \BQ[\si_1,\ldots, \si_m].}
  \end{aligned}
\end{equation}
For example, we have
\[\Bc(\sig_1)=\begin{cases}\si_1 &\text{if $n$ is even} \\\si_1+1&\text{if $n$ is odd.}\end{cases} 
\]

\subsection{Some explicit examples}

In this subsection we make a digression to give some examples of Satake transforms and the base change  homomorphism. Most of it will not be used later. We use for $W_0$ the Hermitian form on the standard vector space given by an anti-diagonal matrix. 

For even $t$ with $0\leq t\leq n$, we introduce the functions
 \begin{equation}\label{def:f}
 f^{[t]}:=\mathbf{1}_{\k\varpi^{(1^{{t}/{2}}, 0^{n-t}, (-1)^{t/2})}\k}\in \hk . 
 \end{equation}
Note that these functions form a polynomial basis of $\CH_K$. We wish to determine their Satake transforms.
 
 We recall some results from \cite{LTXZZ}.  The Langlands dual group  of $\U(W_0)$ is the semi-direct product $\GL_n(\BC)\rtimes \Gal(F/F_0)$. Let $A'$ be the maximal torus containing $A$. Let $\hat A'\subset \GL_n(\BC)$ be the diagonal torus in the dual group of $\GL_n$.  Let $\chi(\rho_{n,s})$ be the restriction of the character of $\wedge^s Std\otimes \wedge^s Std^\vee$ to $\hat A'\times \{\sigma\}\subset \GL_n(\BC)\rtimes \Gal(F/F_0)$, viewed as an element in $\BZ[X^\ast (\hat A')]=\BZ[X_\ast (A')]$. Then    $\chi(\rho_{n,s})\in \BC[X_\ast (A')]^{S_n, \sigma}= \BC[X_\ast (A)]^{W_n}=\BC[\si_1,\cdots,\si_{m}]$. Here we refer to loc.~cit. for the definition of the semi-direct product (defining the L-group of $\U(W_0)$) and of the action of $\sigma$ on the representation space. 
Then  \cite[Lem. B.2]{LTXZZ} (the latter is also \cite[Lem. 9.2.4]{XZ})
$$
\chi(\rho_{n,s})=\begin{cases}  \sum_{j=0}^{[s/2]}  \left(\begin{matrix}m-(s-2j)\\ j \end{matrix}\right) \si_{s-2j}  , & n \text{ even}\\
 \sum_{i=0}^s  \left(\begin{matrix}m-(s-i)\\ [i/2] \end{matrix}\right) \si_{s-i} , & n \text{  odd.} \end{cases}
$$
Set
$$
[n]_q=\frac{q^n-1}{q-1}, \quad [n]_q!= [n]_q [n-1]_q\cdots[1]_q, \quad \left[\begin{matrix}n\\ m \end{matrix}\right]_q=\frac{[n]_{q}!}{[m]_{q}![n-m]_{q}!}
$$
The Satake transforms of $f^{[2s]}$, $1\leq s\leq m$, are determined by the following identity in $\BC[\si_1,\ldots, \si_m]$, cf. \cite[Lem. 2.6]{LTXZZ},
 $$
q^{s(n-s)} \chi(\rho_{n,s})=\sum_{i=0}^s  \left[\begin{matrix}n-2i\\ s -i\end{matrix}\right]_{-q} \Sat(f^{[2i]}),\quad 1\leq s\leq m .
$$
For completeness we also recall  \cite[Lem. B.1.3, B.1.4]{LTXZZ}:
 $$
\prod_{t=1}^m(\lambda+\lambda^{-1}+y_t)=\begin{cases} \chi(\rho_{n,m})+\sum_{i=1}^m \chi(\rho_{n,m-i})( \lambda^i+\lambda^{-i}), & n \text{ even}\\
\sum_{i=0}^m \chi(\rho_{n,m-i})\frac{( \lambda^{i+1}+\lambda^{-i})}{\lambda+1}, & n \text{ odd,} \end{cases}
$$
as an identity of finite Laurent series in $\lambda$.

\begin{example}
  Taking $s=1$ we obtain
  $$
  \left[\begin{matrix}n\\ 1\end{matrix}\right]_{-q}\Sat(f^{[0]})+\left[\begin{matrix}n-2\\ 0\end{matrix}\right]_{-q}\Sat(f^{[2]})=q^{n-1}
  \begin{cases}
    \si_1, & n\text{  even},\\
    \si_1+1, & n\text{  odd}.
  \end{cases}
$$ Therefore $$\Sat(f^{[2]})=-[n]_{-q}+q^{n-1}
  \begin{cases}
    \si_1, & n\text{  even},\\
    \si_1+1, & n\text{  odd}.
  \end{cases}
  $$
  Taking $s=2$ we obtain
  $$
  \left[\begin{matrix}n\\ 2\end{matrix}\right]_{-q}\Sat(f^{[0]})+\left[\begin{matrix}n-2\\ 1\end{matrix}\right]_{-q}\Sat(f^{[2]})+\left[\begin{matrix}n-4\\ 0\end{matrix}\right]_{-q}\Sat(f^{[4]})=q^{2(n-2)}
  \begin{cases}
    \si_2+m, & n\text{  even},\\
    \si_2+\si_1+m, & n\text{  odd}.
  \end{cases}
$$
  
\end{example}

We also describe some explicit functions $\fp^{\prime[t] }$ such that $\Bc(\fp^{\prime[t]})=f^{[t]}$ for even $t$. For $\varphi'\in \Hk$, we view $\varphi'(\alpha_1,\ldots,\alpha_n)$, where $(\alpha_1,\ldots,\alpha_n)\in (\BC^n)^{\rm unit}$, as a symmetric polynomial in $\alpha_i+\alpha_i^{-1}$ for $1\le i\le m$, then the resulting polynomial is nothing but $\Bc(\varphi')$. For example, 
\begin{equation*}
\begin{aligned}
\sigma_2(\alpha_1,\ldots,\alpha_n)=\sum_{1\le i<j\le n}\alpha_i\alpha_j=&\sum_{1\le i<j\le m}(\alpha_i+\alpha_i^{-1})(\alpha_j+\alpha_j^{-1})\\
&+\sum_{1\le i\le m}\alpha_i\alpha_i^{-1} (+\sum_{1\le i\le m}(\alpha_i+\alpha_i^{-1}) \text{ if $n$ is odd}) .
\end{aligned}
\end{equation*} 
Thus 
$$
\Bc(\sig_2)=
  \begin{cases}
    \si_2+m, & n\text{  even},\\
    \si_2+\si_1+m, & n\text{  odd}.
  \end{cases} 
$$
In general, for $1\le s\le m$, we have 
$$
\Bc(\sig_s)=\chi(\rho_{n,s}).
$$ Thus $\{\Sat (f^{[2s]})\}$ can be written as  linear combination of $\{\Bc(\sig_s)\}$ given by
$$
\begin{pmatrix}
  \Sat(f^{[0]})\\
  \Sat(f^{[2]})\\
  \vdots\\
  \Sat(f^{[2m]})
\end{pmatrix}=
\begin{pmatrix}
  \qbin{n}{0} & 0 &\cdots & 0\\
  \qbin{n}{1} & \qbin{n-2}{0} &\cdots & 0\\
  \vdots & \vdots &\ddots &\vdots\\
  \qbin{n}{m} &\qbin{n-2}{m-1} & \cdots & \qbin{n-2m}{0}
\end{pmatrix}^{-1}\cdot
\begin{pmatrix}
  1\\
  q^{n-1}\cdot \Bc(\sig_1)\\
  \vdots\\
  q^{m(n-m)}\cdot\Bc(\sig_m)
\end{pmatrix}
$$

 \subsection{The fundamental lemma, homogeneous version}\label{ss:FL}
We apply the preceding considerations to a split space $W_0$ of dimension $n+1$ and to $W_0^\flat=\langle u_0\rangle^\perp$, where the special vector $u_0\in \Xi$ has unit length. We denote by $K^\flat$ the stabilizer of the selfdual lattice $\Xi^\flat:=W_0^\flat\cap\Xi$. The Haar measures are chosen so that hyperspecial maximal open subgroups have measure one. The Hecke algebra $\CH_{K^\flat\times K}$ for $\U(W_0^\flat)\times\U(W_0)$ can be identified with the tensor product of algebras $\hkf\otimes_\BQ\hk$. The analogous facts hold for the triple $(\GL_n, \GL_{n+1}, \GL_n\times\GL_{n+1})$ and the standard open compact subgroups $(\Kf, \K, \Kf\times\K)$.
We use the same symbol $\Bc$ for the analogous algebra homomorphisms. More precisely, we have $$ \Bc_{n}: \CH_{K^{\prime \flat}}\rightarrow \CH_{K^{ \flat}}, \quad \Bc_{n+1}: \CH_{K'}\rightarrow\hk,\quad\Bc=\Bc_{n}\otimes\Bc_{n+1}: \Hkf\otimes_{\BQ} \Hk\rightarrow\hkf\otimes_{\BQ}\hk.$$

We have the following fundamental lemma for the spherical Hecke algebra in the Jacquet-Rallis case. In it, we let  $W_1$  be the non-split space of dimension $n+1$ and  $W_1^\flat=\langle u_1\rangle^\perp$, where the special vector $u_1$ has unit length. Also, we have chosen bases of $W_0$ and $W_1$ as usual.

 \begin{theorem}[FL for the spherical Hecke algebra (homogeneous version) (Leslie \cite{Les})]

Let $\fp'\in \Hkf \otimes \Hk$. Then 
$$(\Bc(\fp'),0)\in C_c^\infty(G_{W_0})\times C_c^\infty(G_{W_1})$$
 is a transfer of $\fp'$.
\qed
\end{theorem}

\subsection{The fundamental lemma, inhomogeneous version}\label{ss:FL inhom}

Recall the action of $\GL_n$ on $S=S_n$, cf. \eqref{defsy} (except that here $n$ replaces $n+1$) . Let
\[
\CH_{K'_{S_n}}=\CC^\infty_c(S_n(F_0))^{K_n'} .
\]
Note that this is a module over the Hecke algebra $\CH_{K'_n}=\CH_{K'_n}(\GL_n)$. We obtain a map
\begin{equation}\label{eq:act}
{\rm act}:
\CH_{K'_n}\to\CH_{K'_{S_n}}, \quad f'\mapsto f'*{\bf 1}_{K'_{S_n}} .
\end{equation}
Here ${\bf 1}_{K'_{S_n}}$ denotes the characteristic function of $K'_{S_n}=K'_n\cdot 1$. 

There is an alternative description. Let $r: G'\to S_n$ be the map $g\mapsto g \ov g^{-1}$.   
 We also have the map induced by integration on the fibers
 \begin{equation}\label{eq:r *}
 r_\ast:\CH_{K'_n}\to \CH_{K'_{S_n}} ,
 \end{equation}
 which sends $f'$ to $ f^{' \natural}$ defined  by
 $$
 f^{' \natural}(g \ov g^{-1})=\int_{\GL_n(F_0)} f'(gh)\, dh.
 $$ 
 Here we choose the Haar measure on $\GL_n(F_0)$ such that $\vol(\GL_n(O_{F_0}))=1$. Then it is easy to check that the two maps \eqref{eq:act} and  \eqref{eq:r *} coincide, cf. \cite[Lem. 3.3]{Les}.

Using a theorem of Offen \cite{Off}, Leslie shows that
both maps factor through the base change homomorphism $ \Bc=\Bc_{n}: \CH_{K'_n}(\GL_n)\rightarrow \CH_{K}$ and induce an isomorphism with the Hecke algebra for the quasi-split unitary group,
\begin{equation}\label{BCforS}
 \Bc_{S_n}: \CH_{K'_S}\isoarrow \CH_{K} ,
\end{equation}
cf. \cite[Cor. 3.5]{Les}.  We thus have the following commutative diagram,
\begin{equation}\label{eq:diag}
\begin{aligned}
\xymatrix{ \CH_{K'_n} \ar[d]_-{r_\ast} \ar[drr]^-{\Bc_n }& &\\ \CH_{K'_{S_n}} \ar[rr]^-{\sim}_-{ \Bc_{S_n}} 
& & \CH_K } 
\end{aligned}
\end{equation}

For the inhomogeneous FL we need a twist by the algebra automorphism $\eta_\CH: \CH_{K'_n}\to \CH_{K'_n}$ defined by $f\mapsto f\wt \eta$, where $\wt\eta(g)=(-1)^{v(\det(g))}$
  is our fixed extension of the character $\eta$.
 In terms of the Satake isomorphism \eqref{Sat GL}, this automorphism translates to the map $x_i\mapsto -x_i$. In particular, this shows that the map $\eta_\CH$  is an algebra  homomorphism. Via the Satake isomorphism we see that the involution  $\eta_\CH$ descends to $\CH_K$, which we also denote by $\eta_\CH$. We have a commutative diagram $$
\xymatrix{ \CH_{K_n'}\ar[r]^-{\eta_\CH} \ar[d]_-{\Bc_n }& \CH_{K'_n}\ar[d]^{\Bc_n}\\
\CH_K\ar[r]^{\eta_\CH}& \,\CH_K .} 
$$
Let $i\geq 0$. We introduce the $\eta^i$-twisted version of \eqref{BCforS} by 
  \begin{equation}\label{BCforS eta}
  \Bc_{S_n}^{\eta^i}= \eta_\CH^i\circ \Bc_{S_n}:\CH_{K'_S}\isoarrow \CH_{K},
  \end{equation}
  where $\eta_\CH^i$ is the $i$-fold iterate of the automorphism $\eta_\CH$. 
To have a diagram similar to \eqref{eq:diag}, we introduce the $\eta^i$-twist of $r_\ast$:
 \begin{equation}\label{eq:r * eta}
r^{\eta^i}_\ast=r_\ast\circ\eta_\CH^i.
\end{equation} 
Explicitly, we have
 \begin{equation}\label{eq:r eta}
r^{\eta^i}_\ast(f')(g \ov g^{-1})=\int_{\GL_n(F_0)} f'(gh)\wt\eta^i(gh)\, dh.
\end{equation} 
Then we have a  commutative diagram
 \begin{equation}\label{eq:diag1}
 \begin{aligned}
\xymatrix{ \CH_{K'_n} \ar[d]_-{r^{\eta^i}_\ast} \ar[drr]^-{\Bc_n }& & \\
 \CH_{K'_{S_n}}\ar[rr]^{\sim}_-{ \Bc^{\eta^i}_{S_n}}&
& \CH_K} 
\end{aligned}
\end{equation}

There is then  the following inhomogeneous version of the fundamental lemma for the spherical Hecke algebra in the Jacquet-Rallis case. Here we recall that $W_0$ denotes the split hermitian space of dimension $n+1$ and $W_1$ the non-split space (therefore there is a shift of dimensions compared to above). 

 \begin{theorem}[FL for the spherical Hecke algebra (inhomogeneous version) (Leslie \cite{Les})]

Let $\fp'\in \CH_{K'_{S_{n+1}}}$. Then\footnote{Note that in  \cite{Les}, the twist $\eta^{n}$  is erroneously omitted.} 
$$(\Bc^{\eta^{n}}_{S_{n+1}}(\fp'),0)\in C_c^\infty(\U({W_0}))\times C_c^\infty(\U({W_1}))$$
 is a transfer of $\fp'$.
\qed
\end{theorem}

The inhomogeneous version is equivalent to the special case of the  homogeneous version when the factor on $\CH_{K_{n}}$ is the identity ${\bf 1}_{K'_{n}}$. To see this,
we compare orbital integrals: the homogeneous version defined by \eqref{def Orb s} and the inhomogeneous one defined by \eqref{eqn def inhom}. The following easy lemma is a combination of \cite[Lem. 5.7]{RSZ1} and  \cite[Lem. 14.7 (iii)]{RSZ2}. 
\begin{lemma}\label{lem:hom2in}
Let $\fp'\in \Hkf \otimes \Hk$ be of the form ${\bf 1}_{K'^\flat}\otimes f'$ with $ f'\in \Hk$. Then we have, for $\gamma\in S_{n+1}(F_0)_{\rs}$ with $\gamma=r(g)=g\ov g^{-1}$ with $g\in \GL_{n+1}(F)$,
 $$
 \Orb\bigl((1,g),{\bf 1}_{K'^\flat}\otimes f', s\bigr)=\wt\eta^{-n}(g)    \Orb(\gamma,r_{\ast}^{\eta^n}(f'), 2s).
 $$
 Moreover, for $g=(g_1,g_2)\in G'(F_0)_\rs$ and $\gamma=r(g_1^{-1}g_2)\in S(F_0)_\rs$, we have  
  $$
\omega_{G'}(g) \Orb\bigl(g,{\bf 1}_{K'^\flat}\otimes f'\bigr)=\omega_{S}(\gamma)   \Orb(\gamma,r_{\ast}^{\eta^n}(f')),
 $$
 and when $\Orb\bigl(g,{\bf 1}_{K'^\flat}\otimes f'\bigr)=0$, we have
 $$
\omega_{G'}(g) \del\bigl(g,{\bf 1}_{K'^\flat}\otimes f'\bigr)=2\omega_{S}(\gamma)   \del(\gamma,r_{\ast}^{\eta^n}(f')).
 $$
 
\end{lemma}
\begin{proof}

By definition of the orbital integral \eqref{def Orb s}, we have
\[
   \Orb\bigl((1,g),\fp', s\bigr) 
	   = \int_{H_{1, 2}'(F_0)}\fp'(h_1^{-1}h_2',h_1^{-1} g h_2'') \lvert\det h_1\rvert_F^s \eta(h_2) \,dh_1\,dh_2'\,dh_2'',
\]
where $h_1\in H_1'(F_0) = \GL_{n-1}(F)$ and $h_2=(h_2',h_2'')\in H_2'(F_0)=\GL_{n}(F_0)\times\GL_{n+1}(F_0)$. Here  $|\cdot|_F$ is the normalized absolute value on $F$.
Also $\eta(h_2)=\eta^{n-1}(\det h_2')\eta^{n}(\det  h_2'')$. Replacing $h_1$ by $h_2'h_1$, we have 
\[
   \Orb\bigl((1,g),\fp', s\bigr)
	   = \int_{H_{1, 2}'(F_0)} \fp'\bigl(h_1^{-1},h_1^{-1}(h_2')^{-1}g h_2''\bigr) \lvert\det(h_2'h_1)\rvert_F^s \eta(h_2)\,dh_1 \,dh_2'\, dh_2'' .
\]
Now we specialize to $\fp'={\bf 1}_{K'^\flat}\otimes f'$. Then the above equation simplifies to
\[
   \Orb\bigl((1,g),\fp', s\bigr)
	   = \int_{H_{2}'(F_0)} f'\bigl((h_2')^{-1}g h_2''\bigr) \lvert\det(h_2')\rvert_{F_0}^{2s} \eta(h_2) \,dh_2'\, dh_2'' .
\]Here we have used $|a|_F=|a|_{F_0}^2$ for $a\in F_0^\times$ and this results in the extra factor $2$ in the exponent.
To apply the definition of $r_{\ast}^{\eta^n}$, we rewrite it as
\[
   \Orb\bigl((1,g),\fp', s\bigr)
	   =\wt\eta^{-n}(g) \int_{H_{2}'(F_0)} f'\bigl((h_2')^{-1}g h_2''\bigr) \lvert\det(h_2')\rvert_{F_0}^{2s} \eta(h_2')\wt\eta^n((h_2')^{-1}g h_2'')  \,dh_2'\, dh_2'' .
\]
We integrate over $h_2''\in \GL_{n+1}(F_0)$ to
obtain
\[
   \Orb\bigl((1,g),\fp', s\bigr)
	   = \wt\eta^{-n}(g)\int_{H_{2}'(F_0)}  r_{\ast}^{\eta^n}(f')\bigl((h_2')^{-1}(g\ov g^{-})  h_2'\bigr) \lvert\det(h_2')\rvert_{F_0}^{2s} \eta(h_2') \,dh_2'.
\]
A direct comparison with  \eqref{eqn def inhom} completes the proof of the first identity.

For the other identities, 
we compare the transfer factors on $G'$ and $S$. In fact, by their definitions in \cite[\S2.4]{RSZ2} and noting that $\wt\eta$ is of order two by our choice, 
we see that 
$$
  \omega_{G'}(g) = \wt\eta(g_1^{-1}g_2)^n \omega_S\bigl(r(g_1^{-1}g_2)\bigr).
$$
Note that our $n+1$ corresponds to $n$ in loc.~cit. The desired identities follow immediately.

\end{proof}

\section{An alternative basis of $\hk$}\label{s:altbas}
\subsection{An alternative basis of $\hk$}\label{ss:alt basis}
We again denote by $W_0$ a split Hermitian space of dimension $n$.  Let $m=[n/2]$. Fix $\Xi_0$ a self-dual lattice in $W_0$ and let $K=K_0$ be the stabilizer of $\Xi_0$. Recall that a vertex lattice in $W_0$ is a lattice $\Lambda$ such that $\Lambda\subset \Lambda^\vee\subset \varpi^{-1}\Lambda$. Here $\Lambda^\vee$ is the dual lattice of $\Lambda$. The dimension $t=\Lambda^\vee/\Lambda$ is called the type of $\Lambda$. It is an even integer with $0\leq t\leq n$. We fix a maximal chain of vertex lattices
\begin{equation}\label{maxch}
\Xi_0\supset \Xi_2\supset \cdots\supset \Xi_{2m} ,
\end{equation}
where  $\Xi_t$ is a type $t$ vertex lattice, for every even integer $0\leq t\leq n$. 
 Let $K_t\subset \RU(W_0)$ be the stabilizer of $\Xi_t$, a special maximal parahoric subgroup. 
Set
\begin{equation}\label{defphitt'}
\varphi_{[t,t']}={\bf 1}_{K_t K_{t'}}\in \CH(K_t\bs \RU(W_0)/ K_{t'}),
\end{equation}
which will be called the \emph{intertwining function} for the level $K_t,K_{t'}$. (See \cite[Definition B.2.3]{LTXZZ} for the special case $t=t'=2m$.) Note that 
${\bf 1}_{K_t K_{t'}}$ is the characteristic function of a subset of $\RU(W_0)$ that is not a subgroup, unless $K_t =K_{t'}$.

Note that a function $\varphi\in C_c^{\infty}(K\bs G/K')$ induces a natural map $C_c^\infty(G/ K)\to C_c^\infty(G/ K')$ which determines the function uniquely. Explicitly, a given $\varphi$ sends $f\in C_c^\infty(G/ K)$ to $\frac{1}{\vol(K)} f\ast\varphi$  which lies in $ C_c^\infty(G/ K')$ (the volume factor is inserted to be compatible with the ``moduli interpretation" below).  Similarly, a correspondence between the two sets $G/ K$ and $G/ K'$ also induces a natural map $C_c^\infty(G/ K)\to C_c^\infty(G/ K')$. In this way we will freely switch between functions and correspondences. We can identify the function $\varphi_{[t, t']}$ in terms of  ``Hecke correspondences of moduli spaces of vertex lattices", as follows.  Let  
$$
\BN^{[t]}=\BN_{W_0}^{[t]}=\{\Lambda\subset W_0\mid \text{$\Lambda$ is a vertex lattice of type $t$}\}.
$$
Note that there is the diagonal action of $\U(W_0)$ on $\BN^{[t]}\times \BN^{[t']}$. Then $\CH(K_t\bs \RU(W_0)/ K_{t'})$ can be identified with the space of functions on  $\BN^{[t]}\times \BN^{[t']}$ which are invariant under $\U(W_0)$ and have compact support modulo this action.

We also introduce  
\begin{equation*}
\BN^{[t,t']}=\BN^{[t,t']}_{W_0}=\begin{cases}\{(\Lambda_t,\Lambda_{t'})\in \BN^{[t]}\times \BN^{[t']}\mid \Lambda_t\subset \Lambda_{t'} \} & \text{if $t'\leq t$}\\
\{(\Lambda_t,\Lambda_{t'})\in \BN^{[t]}\times \BN^{[t']}\mid \Lambda_{t'}\subset \Lambda_{t} \} & \text{if $t\leq t'$.}
\end{cases}
\end{equation*}
We obtain a diagram,
\begin{equation}\label{Lattt'}
\begin{aligned}
\xymatrix{&\BN^{[t, t']} \ar[rd]  \ar[ld] & \\ \BN^{[t]}  & &\BN^{[t']} , } 
\end{aligned}
\end{equation}
or, equivalently, the map
\[
\pi\colon \BN^{[t, t']}\to \BN^{[t]}\times \BN^{[ t']} .
\]
Then 
\begin{equation}\label{pitt'}
\varphi_{[t, t']}=\pi_*(\rm {char}_{\BN^{[t, t']}}) .
\end{equation}

Note that $\varphi_{[t, t']}$ does not lie in the spherical Hecke algebra. To obtain functions in $\hk$, we use convolution.
 \begin{definition}\label{def:varphi t}
 The atomic function associated to  an even integer $t$ with $0\leq t\leq n$ is the following element in the spherical Hecke algebra, 
\begin{align}
\varphi_{t}=\vol(K_0)^{-1}\vol(K_t)^{-1}\varphi_{[0,t]}\ast  \varphi_{[t,0]}=\vol(K_t)^{-1}{\bf 1}_{K_0 K_{t}} \ast {\bf 1}_{K_t K_{0}}\in \CH_\k ,
\end{align}
where we recall that $\vol(K_0)=1$.
\end{definition}

Consider the composite of correspondences,
$$
\BT_{\leq t}=\BT_{W_0, \leq t}:=\BN^{[0, t]}\circ \BN^{[t,0]}=\{(\Lambda_0,\Lambda_{t},\Lambda_0')\in \BN^{[0]}\times \BN^{[t]}\times  \BN^{[0]}\mid \Lambda_t\subset \Lambda_{0}\cap \Lambda_{0}' \}.
$$ 
We obtain a diagram with a cartesian square,
\begin{equation}
\begin{aligned}
\xymatrix{&&\BT_{\leq t} \ar[rd]  \ar[ld] & \\ &\BN^{[0,t]} \ar[rd]  \ar[ld] & &\BN^{[t,0]} \ar[rd]  \ar[ld] &\\ \BN^{[0]} &&  \BN^{[t]}&&  \BN^{[0]}.}
\end{aligned}
\end{equation}
Consider
$$
\pi: \BT_{\leq t}\to \BN^{[0]}\times\BN^{[0]} .
$$
Then, just as in  \eqref{pitt'}, 
\begin{equation}
\varphi_{t}=\pi_*({\rm char}_{\BT_{\leq t}}).
\end{equation}

\begin{remark}\label{commH}
As is well-known, the spherical Hecke algebra is commutative. In particular, $\varphi_t*\varphi_{t'}=\varphi_t*\varphi_{t'}$. Comparing the supports of these two functions, we get the equality of the following two subsets of $\BN^{[0]}\times\BN^{[0]}$. 

We say that two selfdual lattices $\Lambda_1$ and $\Lambda_2$ are related by a correspondence of type $t$, denoted $\Lambda_1\!\!\iff\!\!\!^t\,\,\Lambda_2$, if there exists a vertex lattice $M$ of type $t$ contained in $\Lambda_1\cap\Lambda_2$. In other words, $(\Lambda_1, \Lambda_2)\in \pi(\BT_{\leq t})$. The two sets in question are 
\begin{equation}\label{eqcommH}
\begin{aligned}
&\{(\Lambda_1, \Lambda_2)\mid \exists \Lambda\in\BN^{[0]} \text{ with $\Lambda_1\!\!\iff\!\!\!^t\,\,\Lambda$ and $ \Lambda\!\!\iff\!\!\!^{t'}\,\,\Lambda_2$}\}.\\
&\{(\Lambda_1, \Lambda_2)\mid \exists \Lambda\in\BN^{[0]} \text{ with $\Lambda_1\!\!\iff\!\!\!^{t'}\,\,\Lambda$ and $ \Lambda\!\!\iff\!\!\!^{t}\,\,\Lambda_2$}\}.
\end{aligned}
\end{equation}
Here is an elementary proof of  the equality of these two sets that was indicated to us by B.~Howard. It is based on Gelfand's trick  of proving the commutativity of the Hecke algebra by constructing an anti-automorphism of $\U(W_0)$ inducing the identity on $K\backslash\U(W_0)/K$. Fix a basis of $W_0$ such that the hermitian form is given by the antidiagonal unit matrix $J$, and let $\Xi_0$ be the standard lattice. Define the anti-automorphism $\tau$ of $\U(W_0)$ by $\tau(g)=J\sigma(g^{-1})J$. Then $\tau$ induces the identity on $A(F_0)$ and hence induces the identity on $\CH_K$ (Cartan decomposition). On the other hand,  one easily checks that under the identification $K\backslash\U(W_0)/K=\U(W_0)\backslash (\BN^{[0]}\times\BN^{[0]}) $, the map $\tau$ is given by $(\Lambda_0, \Lambda_0')\mapsto (\sigma(\Lambda_0'), \sigma(\Lambda_0))$. It follows that for any $(\Lambda_0, \Lambda_0')\in \BN^{[0]}\times\BN^{[0]}$ there exists $\gamma\in\U(W_0)(F_0)$ such that $(\sigma(\Lambda_0'), \sigma(\Lambda_0))=\gamma(\Lambda_0, \Lambda_0')$. Let us now check the claim. By symmetry, it suffices to show that the first of the two sets in \eqref{eqcommH} is contained in the second. Let  $(\Lambda_1, \Lambda_2)$ be in the first  set, i.e., there is $\Lambda\in\BN^{[0]}$  with $\Lambda_1\!\!\iff\!\!\!^t\,\,\Lambda$ and $ \Lambda\!\!\iff\!\!\!^{t'}\,\,\Lambda_2$. Let $(\sigma(\Lambda_2), \sigma(\Lambda_1))=\gamma(\Lambda_1, \Lambda_2)$. Then $\sigma(\Lambda_2)\!\!\iff\!\!\!^t\,\,\gamma\Lambda$ and $ \gamma\Lambda\!\!\iff\!\!\!^{t'}\,\,\sigma(\Lambda_1)$. It follows that $(\Lambda_1, \Lambda_2)$ is in the second set, with intermediary lattice $\Lambda'=\sigma(\gamma\Lambda)$.
\end{remark}

Besides $\BT_{\leq t}$, we also consider 
\begin{equation*}
\begin{aligned}
\BT_{t}=\BT_{W_0, t}&=\{(\Lambda_0,\Lambda_0')\in \BN^{[0]}\times \BN^{[0]}\mid \Lambda_{0}\cap \Lambda_{0}' \text{ is a vertex lattice of type $t$}\}\\
 &=\{(\Lambda_0,\Lambda_{t},\Lambda_0')\in \BN^{[0]}\times \BN^{[t]}\times  \BN^{[0]}\mid \Lambda_t= \Lambda_{0}\cap \Lambda_{0}' \},
\end{aligned}
\end{equation*}
with its natural map
\[
\pi: \BT_{ t}\to \BN^{[0]}\times\BN^{[0]} .
\]
Recall $f^{[t]}:=\mathbf{1}_{\k\varpi^{(1^{t/2}, 0^{n-2t}, (-1)^{t/2})}\k}$, cf. \eqref{def:f}. Then 
\begin{equation}
f^{[t]}=\pi_*({\rm char}_{\BT_{t}}).
\end{equation}

\begin{proposition}\label{lem:basis}
The  spherical Hecke algebra $\CH_K$ is a polynomial algebra  in the atomic functions $\varphi_t$, as $t$ runs through all even integers $ 0\leq t\leq n$, 
$$
\CH_K=\BQ[\varphi_2,  \varphi_4,\cdots, \varphi_{2m}].
$$
More precisely, the elements $\{\varphi_{t}\}_{0<t\leq n,  t\equiv 0\!\!\mod 2}$ are expressed as a linear combination of $\{f^{[t]}\}_{0<t\leq n,  t\equiv 0\!\!\mod 2}$ by an upper-triangular matrix with all diagonal entries equal to one.
\end{proposition}
\begin{proof}
From the ``moduli interpretation", we can partition  $\BT_{\leq t}$ into a disjoint union  
$$\BT_{\leq t}=\coprod_{t'\leq t,\atop t'\equiv 0\!\!\mod 2}\BT_{\leq t}^{[t']} ,$$ 
where
$$
\BT_{\leq t}^{[t']}=\{(\Lambda_0,\Lambda_{t},\Lambda_0')\in \BT_{\leq t} \mid \text{ $\Lambda_{0}\cap \Lambda_{0}' $ has type $t'$}\}.
$$It follows that
\begin{equation}\label{ftophi}
\varphi_t=\sum_{t'\leq t,\atop t'\equiv 0\!\!\mod 2 }  m(t',t)f^{[t]} ,
\end{equation}
where $m(t',t)$ is the number of vertex lattices of type $t$ contained in a given vertex lattice of type $t'$,
$$
m(t',t)=\#\{ \Lambda_t\in \BN^{[t]}\mid \Lambda_t\subset \Lambda_{t'}\}.
$$
That is, $m(t',t)$ equals the (constant) degree of the fiber of the natural projection map $ \BN^{[t,t']}\to  \BN^{[t']}$. Clearly $$
m(t,t)=1.$$
This shows that the basis $\{\varphi_{t}\}_{0<t\leq n,  t\equiv 0\!\!\mod 2}$ differs from $\{f^{[t]}\}_{0<t\leq n,  t\equiv 0\!\!\mod 2}$ by an upper-triangular matrix with all diagonal entries equal to one. This implies the desired assertion.

\end{proof}
\begin{remark}One can determine the coefficients $m(t',t)$ explicitly. 
See
\cite[Lem. B.2.4]{LTXZZ} for the formula expressing $  \varphi_{2m}$ in terms of the $f^{[t]}$.
\end{remark}

\begin{definition}\label{def:basic fun}
We call an element $\varphi\in \CH_K$ monomial   if it can be expressed as a monomial in the atomic functions $\{\varphi_2,  \varphi_4,\cdots, \varphi_{2m}\}$ (in a necessarily unique way).
\end{definition}
It is clear that an arbitrary element $\varphi\in \CH_K$ can be expressed as a linear combination of monomial elements  (in a  unique way).
\subsection{The product situation}
 We also apply the preceding considerations  to a product situation. Let $W_0$ be a split space of dimension $n+1$, and  let $W_0^\flat=\langle u\rangle^\perp$, where the special vector $u\in W_0$ has unit length. Then $W_0^\flat\oplus W_0$ is also a split space. We fix a maximal chain of vertex lattices in $W_0$ as in \eqref{maxch} (where $n$ is replaced by $n+1=\dim W_0$) 
  \begin{equation}\label{maxch0}
\Xi_0\supset \Xi_2\supset \cdots\supset \Xi_{2m} ,
\end{equation}
and  a maximal chain in $W_0^\flat$
 \begin{equation}\label{maxchflat}
\Xi^\flat_0\supset \Xi^\flat_2\supset \cdots\supset \Xi^\flat_{2m^\flat} ,
\end{equation}
where  $\Xi^\flat_t$ is a type $t$ vertex lattice in $W_0^\flat$, for every even integer $0\leq t\leq n$. Here $m=\lfloor{(n+1)/2}\rfloor $ and $m^\flat=\lfloor{n/2}\rfloor$. We assume that the self-dual lattice $\Xi_0$  contains the special vector $u_0$, and that $\Xi^\flat_0$ is related by $
\Xi_0=\Xi^\flat_0\obot \pair{u_0}.$ Note that we only assume that $u_0\in\Xi_0$, while other lattices $\Xi_t$ ($t>0$) may not contain $u_0$. Then the different choices of $\Xi^\flat_t$ may yield different parahoric subgroups $K^\flat_t$, but the resulting atomic function $\varphi_t$ is independent of the choices. The Hecke algebra $\CH_{K^\flat\times K}$ for $\U(W_0^\flat)\times\U(W_0)$ can be identified with the tensor product of algebras $\hkf\otimes\hk$. By an atomic element in $\CH_{K^\flat\times K}$ we mean a pure tensor $\wt\varphi=\varphi^\flat\otimes\varphi$, where either $\varphi^\flat$ is atomic and $\varphi$ is the unit element, or $\varphi$ is  atomic  and $\varphi^\flat$ is the unit element. A monomial element in $\CH_{K^\flat\times K}$ is defined to be a product of atomic elements. This presentation is  unique. Any element in $\CH_{K^\flat\times K}$ is a linear combination of monomial elements in a unique way.
 
 Let $\varphi_{t, t'}=\varphi_t\otimes\varphi_{t'}$ be an atomic function. Hence either $t=0$ or $t'=0$. Then
 \begin{equation}
 \varphi_{t, t'}=\pi_*( {\rm char}_{\BT_{W_0^\flat,\leq t}\times\BT_{W_0,\leq t'}}).
 \end{equation}
 Here we use the diagram 
\begin{equation}
\begin{aligned}
\xymatrix{&\BT_{W_0^\flat,\leq t}\times\BT_{W_0,\leq t'}\ar[rd]  \ar[ld] & \\ \BN_{W_0^\flat}^{[0]}\times \BN_{W_0}^{[0]}  & &\BN_{W_0^\flat}^{[0]}\times \BN_{W_0}^{[0]} , } 
\end{aligned}
\end{equation}
and the corresponding map
\[
\pi\colon \BT_{W_0^\flat,\leq t}\times\BT_{W_0,\leq t'}\to  (\BN_{W_0^\flat}^{[0]}\times\BN_{W_0}^{[0]})\times (\BN_{W_0^\flat}^{[0]}\times\BN_{W_0}^{[0]}) .
\]
The following lemma holds for all functions $\varphi\in C_c^\infty(G_{W_0})$ (and is used earlier in the definition of transfer). We include a proof for the special case of functions in the Hecke algebra, in order to illustrate the idea that will be used in the proof of Proposition \ref{prop:finite}.
\begin{lemma}\label{lem:orb finite}
Let $\varphi\in \hkf\otimes\hk$. Then $\Orb(g,\varphi)$ is finite for every $g\in G_W(F)_{\rs}$.
\end{lemma}
\begin{proof}It suffices to consider monomial functions and elements  of the form $(1,g)$ with $g\in \RU(W_0)_\rs$.
We can bound $\varphi$ by a constant multiple of a function of the following form $$
\Phi_N={\bf 1}_{\RU(W_0^\flat)\cap \varpi^{-N}\End_{O_F}(\Xi^\flat_0)}\otimes {\bf 1}_{\RU(W_0)\cap \varpi^{-N}\End_{O_F}(\Xi_0)}$$
for some large integer $N$, and hence  it suffices to consider such functions $\Phi_N$. In terms of lattice counting, we need to show the finiteness of pairs of self-dual lattices 
$$
(\Lambda^\flat, \Lambda), \quad (\Lambda'^\flat, \Lambda')
$$
such that $\Lambda=\Lambda^\flat\oplus \pair{u}, \Lambda'=\Lambda'^\flat\oplus \pair{u}$, and $\Lambda^\flat \subset \varpi^{-N} \Lambda'^\flat$, and $\Lambda \subset \varpi^{-N} g \Lambda'$. Note that $\Lambda^\flat$ (resp. $\Lambda'^\flat$) is determined by $\Lambda$ (resp. $\Lambda'$). The self-duality also implies 
 $\Lambda^\flat \supset \varpi^{N} \Lambda'^\flat$ and $\Lambda \supset \varpi^{N} g \Lambda'$.

 We claim that {\em $\varpi^{(2i-2)N}g ^{i-1}u\in \Lambda'$ and $\varpi^{(2i-1)N}g ^iu\in \Lambda $ for every $i\geq 1$.}  To show the claim, it suffices to show 
\begin{enumerate}
\item  $u\in\Lambda'$, 
\item  $\varpi^{(2i-2)N}g ^{i-1}u\in \Lambda'\imp \varpi^{(2i-1)N}g ^iu\in \Lambda $ for  every $i\geq 1$,
\item  $\varpi^{(2i-1)N}g ^{i}u\in \Lambda \imp \varpi^{2iN}g ^iu\in \Lambda '$ for  every $i\geq 1$.
\end{enumerate}

\noindent \emph{Ad }(1):  clear. 

\noindent \emph{Ad}  (2): if $\varpi^{(2i-2)N}g ^{i-1} u\in \Lambda'$ then, from $\Lambda \supset \varpi^{N} g \Lambda'$, it follows that  $\varpi^{N} g (\varpi^{(2i-2)N}g ^{i-1} u)=\varpi^{(2i-1)N}g ^i u\in \Lambda$. 

\noindent \emph{Ad}  (3): if $\varpi^{(2i-1)N}g ^{i}u\in \Lambda$ then, 
from  $\Lambda^\flat \subset \varpi^{-N} \Lambda'^\flat$, it follows that  $  \Lambda'\supset  \varpi^{N} \Lambda$ and hence $ \varpi^{N}(\varpi^{(2i-1)N}g ^{i}u)= \varpi^{2iN} g^i u\in \Lambda'$.

 From the claim it follows that $\Lambda$ contains the lattice  $\varpi^{(2n-1)N}\pair{u,gu, \ldots, g^n u}$, which has full rank by the regular semisimplicity of $g$, cf. \cite[\S 2.4]{RSZ1}. This shows the finiteness of possible $\Lambda$, and hence of $\Lambda'$.  This completes the proof.
\end{proof}

\section{RZ spaces and their Hecke correspondences}\label{s:RZplusHeck}

In this section we define various Rapoport--Zink spaces (RZ spaces) with certain parahoric levels, and use them to define Hecke correspondences.
\subsection{Rapoport--Zink spaces $\Nn$ of self-dual level} \label{sec:rapoport-zink-spaces}

Let $S$ be a $\Spf \OFb$-scheme. Consider a triple $(X, \iota,\lambda)$ where
\begin{altenumerate}
\item $X$ is a formal $\varpi$-divisible $O_{F_0}$-module over $S$ of relative height $2n$ and dimension $n$,
\item $\iota: O_F\rightarrow\End(X)$ is an action of $O_F$ extending the $O_{F_0}$-action and satisfying the Kottwitz condition of signature $(1,n-1)$: for all $a\in O_F$, the characteristic polynomial of $\iota(a)$ on $\Lie X$ is equal to $(T-a)(T-\sigma(a))^{n-1}\in \mathcal{O}_S[T]$,
\item $\lambda: X\rightarrow X^\vee$ is a principal polarization on $X$ whose Rosati involution induces the automorphism $\sigma$ on $O_F$ via $\iota$.
\end{altenumerate}

Up to $O_F$-linear quasi-isogeny compatible with polarizations, there is a unique such triple $(\mathbb{X}, \iota_{\mathbb{X}}, \lambda_{\mathbb{X}})$ over $S=\Spec \kb$. Let $\Nn=\mathcal{N}_{F/F_0, n}$ be the (relative) \emph{unitary Rapoport--Zink space of self-dual level}, which is a formal scheme over $\Spf\OFb$ representing the functor sending each $S$ to the set of isomorphism classes of tuples $(X, \iota, \lambda, \rho)$, where the \emph{framing} $\rho: X\times_S \bar S\rightarrow \mathbb{X}\times_{\Spec \kb}\bar S$ is an $O_F$-linear quasi-isogeny of height 0 such that $\rho^*((\lambda_\mathbb{X})_{\bar S})=\lambda_{\bar S}$. Here $\bar S\coloneqq S_{\kb}$ is the special fiber.

The Rapoport--Zink space $\Nn$ is formally locally of finite type and formally smooth of relative dimension $n-1$ over $\Spf \OFb$ (\cite{RZ96}, \cite[Prop. 1.3]{Mihatsch2016}).

\subsection{The hermitian space $\mathbb{V}_n$}\label{sec:herm-space-mathbbv}

Let $n\ge1$ be an integer. Let $\mathbb{E}$ be the formal $O_{F_0}$-module of relative height 2 and dimension 1 over $\Spec \kb$. Then $D\coloneqq \End_{O_{F_0}}^\circ(\mathbb{E}):=\End_{O_{F_0}}(\mathbb{E}) \otimes \mathbb{Q}$ is the quaternion division algebra over $F_0$. We fix an $F_0$-embedding $\iota_\mathbb{E}:F\rightarrow D$, which makes $\mathbb{E}$ into a formal $O_F$-module of relative height 1. We fix an $O_{F_0}$-linear principal polarization $\lambda_\mathbb{E}: \mathbb{E}\xrightarrow{\sim} \mathbb{E}^\vee$. Then $(\mathbb{E}, \iota_\mathbb{E},\lambda_\mathbb{E})$ is a hermitian $O_F$-module of signature $(1,0)$. We have $\mathcal{N}_1\simeq \Spf \OFb$ and there is a unique lifting (\emph{the canonical lifting}) $\mathcal{E}$ of the formal $O_F$-module $\mathbb{E}$ over $\Spf \OFb$, equipped with its $O_F$-action $\iota_\mathcal{E}$, its framing $\rho_\mathcal{E}: \mathcal{E}_{\kb}\xrightarrow{\sim}\mathbb{E}$, and its principal polarization $\lambda_\mathcal{E}$ lifting $\rho_\mathcal{E}^*(\lambda_\mathbb{E})$. Define $\barE$ to be the same $O_{F_0}$-module as $\mathbb{E}$ but with $O_F$-action given by $\iota_{\barE}\coloneqq \iota_\mathbb{E}\circ \sigma$, and $\lambda_{\barE}\coloneqq \lambda_{\mathbb{E}}$, and similarly define $\bar{\mathcal{E}}$ and $\lambda_{\bar{\mathcal{E}}}$.

Denote by $\mathbb{V}=\mathbb{V}_n:=\Hom_{O_F}^\circ(\barE, \mathbb{X})$  the space of special quasi-homomorphisms. Then $\mathbb{V}$ carries a $F/F_0$-hermitian form: for $x,y\in \mathbb{V}$, the pairing $(x,y)\in F$ is given by the composition 
$$ (\barE\xrightarrow{x} \mathbb{X}\xrightarrow{\lambda_\mathbb{X}} {\mathbb{X}^\vee}\xrightarrow{y^\vee} \barE^\vee\xrightarrow{\lambda_\mathbb{E}^{-1}}\barE)\in\End_{O_{F}}^\circ(\barE)=\iota_\barE(F)\simeq F.$$
The hermitian space $\mathbb{V}$ is the unique (up to isomorphism) non-degenerate non-split $F/F_0$-hermitian space of dimension $n$. The unitary group $\U(\mathbb{V})(F_0)$ acts on the framing hermitian $O_F$-module $(\mathbb{X}, \iota_\mathbb{X},\lambda_\mathbb{X})$ (via the identification in \cite[Lem. 3.9]{Kudla2011}) and hence acts on the Rapoport--Zink space $\mathcal{N}_n$ via $g(X,\iota,\lambda,\rho)=(X, \iota,\lambda, g\circ\rho)$ for $g\in \U(\mathbb{V})(F_0)$.

To a non-zero $u\in\BV$, there is the associated special divisor $\CZ(u)$ on $\CN_n$. It is the closed formal subscheme of $\CN_n$ of points $(X, \iota, \lambda, \rho)$ where the quasi-homomorphism $u: \BE\to \BX$ lifts to a homomorphism $\CE\to X$, cf. \cite{Kudla2011}. More generally, for any finitely generated subgroup $L\subset\BV$, there is the associated \emph{special cycle} $\CZ(L)$, the locus where the quasi-homomorphisms $u: \BE\to \BX$ lift to  homomorphisms $\CE\to X$, for all $u\in L$.
\subsection{RZ spaces    $\CN_n^{[t]}$}\label{ssRZtt'}
 Let $t$ be an even integer with $0\leq t\leq n$. Now we consider triples   $(Y, \iota_Y,\lambda_Y)$ as in subsection \ref{sec:rapoport-zink-spaces} (with $Y$ instead of $X$), except that we replace the condition on the polarization to be principal by the condition that the  polarization is of degree $q^{2t}$ and satisfies $\ker(\lambda_Y)\subset Y[\varpi]$. Again, we fix a framing object $(\BX_t, \iota_{\BX_t}, \lambda_{\BX_t})$ over $\Spec \bar k$ (which is again unique up to $O_F$-linear quasi-isogeny compatible with polarizations). 
 We define the RZ-space $\CN_n^{[t]}$ as the space of tuples $(Y, \iota_Y,\lambda_Y, \rho_Y)$, where $\rho_Y$ is a framing with $(\BX_t, \iota_{\BX_t}, \lambda_{\BX_t})$.  It is an RZ space of level equal to  the stabilizer of a vertex lattice of type $t$.  Note that $\CN^{[0]}_n=\CN_n$. The Rapoport--Zink space $\Nn^{[t]}$ is formally locally of finite type and regular of dimension $n$ with semi-stable reduction over $\Spf \OFb$ (\cite{RZ96}, \cite{Go}).

\subsection{RZ spaces   $\CN_n^{[t,t']}$}
Let $t, t'$ be  even integers. 
We introduce the RZ-space $\CN_n^{[t, t']}$. Let first  $t'\leq t$. We fix a $O_{F_0}$-linear isogeny of degree $q^{t-t'}$ with kernel killed by $\varpi$,
\begin{equation}
\varphi_{t, t'}\colon \BX_{t}\to \BX_{t'} ,
\end{equation}
such that $\varphi^*(\lambda_{\BX_{t'}})=\lambda_{\BX_t}$.   Then $\CN_n^{[t, t']}$ classifies triples
$$
\big((X_1, \iota_1,\lambda_1, \rho_1), (X_2, \iota_2,\lambda_2, \rho_2), \varphi:X_1\to X_2\big) ,
$$
where $(X_1, \iota_1,\lambda_1, \rho_1)\in \CN^{[t]}$ and $(X_2, \iota_2,\lambda_2, \rho_2)\in\CN^{[t']}$, and where $\varphi$ is an  isogeny lifting $\varphi_{t, t'}\colon \BX_{t}\to \BX_{t'}$. Then $\varphi$ is  uniquely determined, is of degree $q^{t-t'}$ and satisfies $\ker\varphi\subset X_1[\varpi]$; also,  $\varphi$ preserves the polarizations.

 When $t\leq t'$, we define  $\CN_n^{[t, t']}$ to be the \emph{transpose} $^t\CN_n^{[t', t]}$ of   $\CN_n^{[t', t]}$, i.e., $\CN_n^{[t, t']}$ classifies triples 
$$
\big((X_1, \iota_1,\lambda_1, \rho_1),  (X_2, \iota_2,\lambda_2, \rho_2),  \varphi:X_2\to X_1\big) ,
$$
where $(X_1, \iota_1,\lambda_1, \rho_1)\in \CN^{[t]}$ and $(X_2, \iota_2,\lambda_2, \rho_2)\in\CN^{[t']}$, and where $\varphi$ is an   isogeny lifting $\varphi_{t', t}\colon \BX_{t'}\to \BX_{t}$. Only the cases $(t,t')=(0,t')$ and $(t,t')=(t,0)$ will be effectively  used in this paper.  We view $\CN_n^{[t,t']}$ as a correspondence via the two natural projections, which is  analogous to the corresponding diagram \eqref{Lattt'} of lattice correspondences,
\begin{equation}\label{RZtt'}
\begin{aligned}
\xymatrix{&\CN_n^{[t,t']}\ar[ld]_{\pi_1} \ar[rd]^{\pi_2}  & \\ \CN_n^{[t]}  & &\CN_n^{[t']} .  }
\end{aligned}
\end{equation}
The two projection maps are both proper (representable by a projective morphism). Note that  $\CN_n^{[t,t']}$ is an RZ space of level equal to  the joint stabilizer of two vertex lattices of type $t, t'$.  It  is formally locally of finite type and regular of dimension $n$ with semi-stable reduction over $\Spf \OFb$ (\cite{RZ96}, \cite{Go}).

\subsection{The Hecke correspondence $\CT_n^{\leq t}$}

Define $\CT_n^{\leq t}$ by
$$
\big((Y, \iota_Y,\lambda_Y, \rho_Y),(X_1, \iota_1,\lambda_1, \rho_1), (X_2, \iota_2,\lambda_2, \rho_2),  \varphi_i: Y\to X_i, i=1,2\big)
$$
such that $(Y,X_i,\varphi_i)\in \CN_n^{[t,0]}$. In other words, \ $\CT_n^{\leq t}$ is the composition of correspondences,  
$$
\CT_n^{\leq t}= \CN_n^{[0,t]}\circ \CN_n^{[t,0]},
$$
cf. subsection \ref{compofcorr}. 
 Explicitly, this means that we obtain the following diagram with a cartesian square in the middle, 
\begin{equation}\label{diagCN}
\begin{aligned}
\xymatrix{&&\CT_n^{\leq t} \ar[rd]  \ar[ld] & \\ &\CN_n^{[0,t]} \ar[rd]  \ar[ld] & &\CN_n^{[t,0]} \ar[rd]  \ar[ld] &\\ \CN^{[0]}_n &&  \CN_n^{[t]}&&  \CN^{[0]}_n.}
\end{aligned}
\end{equation}
Note that  all formal schemes in the lower  two rows  are regular and the maps are relatively representable by morphisms of projective schemes. Using the calculus in the appendix, we may   therefore 
define for closed formal subschemes $A\subset \CN_n^{[0]}$, resp. $B\subset \CN^{[t]}_n$ 
\begin{equation}
\begin{aligned}
   &\BT_t^{A, +}=(\CN^{[0,t]}_n)_*\colon K^A(\CN^{[0]}_n)\to K^{\CN^{[0,t]}_n(A)}(\CN^{[t]}_n), \\
   &\BT_t^{B, -}=(\CN^{[t, 0]}_n)_*\colon K^B(\CN^{[t]}_n)\to K^{\CN^{[t, 0]}_n(B)}(\CN^{[0]}_n) . 
\end{aligned}
\end{equation}  
This defines the map   
\begin{equation}\label{elhec}
\BT^A_t= \BT^{\CN^{[0,t]}_n(A), -}_t\circ\BT^{A, +}_t\colon K^A(\CN_n)\to K^{\CT^{\leq t}(A)}(\CN_n) .
\end{equation}
Recall here that $\CN_n=\CN_n^{[0]}$.  

We state the  following conjecture. Note that this conjecture is empty for $n=2$ and $n=3$. 
\begin{conjecture}\label{conjcom}
Let $t\neq t'$. Then  for any closed formal subscheme $A$ of $\CN_n$  we have an equality of the two closed formal subschemes $\CT^{\leq t}(\CT^{\leq t'}(A))$ and $\CT^{\leq t'}(\CT^{\leq t}(A))$ of $\CN_n$. Furthermore, the two maps
\begin{equation*}
\BT^{\CT^{\leq t'}(A)}_t\circ\BT^A_{t'}\colon K^A(\CN_n)\to K^{\CT^{\leq t}(\CT^{\leq t'}(A))}(\CN_n), \quad \BT^{\CT^{\leq t}(A)}_{t'}\circ\BT^A_{t}\colon K^A(\CN_n)\to K^{\CT^{\leq t'}(\CT_{t}(A))}(\CN_n)
\end{equation*} 
are identical modulo torsion.
\end{conjecture}

\begin{remark}
The conjectured equality $\CT^{\leq t}(\CT^{\leq t'}(A))=\CT^{\leq t'}(\CT^{\leq t}(A))$ is analogous to the identity of the two subsets \eqref{eqcommH} of $\BN^{[0]}\times \BN^{[0]}$ in Remark \ref{commH}. 
\end{remark}

For  an atomic function $\varphi_t$ (cf. Definition \ref{def:varphi t}) with $0\leq t\leq n$, we define the corresponding Hecke operator as
$$
\BT^A_{\varphi_t}:=\BT^A_t\colon K^A(\CN_n)\to K^{\CT^{\leq t}(A)}(\CN_n).
$$
The homomorphisms obtained in this way will be called  \emph{atomic Hecke operators}.  By Proposition \ref{lem:basis}, atomic functions form a basis of the spherical Hecke algebra (as a polynomial algebra). 
Following \cite[Lem. 1.4]{Gillet1987}, we introduce 
\[
K^\sigma(\CN)_\BQ=\oplus_{A\subset \CN_n}\, K^A(\CN_n)\otimes\BQ ,
\]
where $A$ runs through all closed formal subschemes defined by radical ideal sheaves. Then we obtain for every $t$ an endomorphism $\BT_{\varphi_t}$ of $K^\sigma(\CN_n)_\BQ$ by mapping an element $(\gamma_A)_A$ to $(\beta_B)_B$, where 
\[
\beta_B=\sum_{\{A\mid \CT^{\leq t}(A)\equiv B\}}\gamma_A  .
\] 
Here, the notation means that the radical of the ideal defining the closed formal subscheme $\CT^{\leq t}(A)$ coincides with the ideal sheaf of $B$.

Denoting by $\wt\CH_K$ the non-commutative polynomial algebra over  $\BQ$ in the indeterminates  $X_1,\ldots,X_m$, we therefore obtain by taking compositions a homomorphism of algebras
\begin{equation}\label{wthomo}
\BT\colon \wt\CH_K\to\End(K^\sigma(\CN_n)_\BQ ).
\end{equation}
Explicitly, if $\varphi=\varphi_1*\varphi_2*\cdots *\varphi_r$ is a monomial in the atomic Hecke functions $\varphi_1=\varphi_{t_1}, \varphi_2=\varphi_{t_2},\ldots,\varphi_r=\varphi_{t_r}$, then $\BT^A_\varphi={\BT^{\CT_{\leq t_{2}}\circ\CT_{\leq t_{3}}\circ\ldots\circ\CT_{\leq t_{r}}(A) }_{t_1}}\circ\cdots \circ{\BT_{\leq t_{r-1}}^{\CT_{\leq t_r}(A)}}\circ{\BT_{\leq t_r}^A}$. These  endomorphisms of $K^\sigma(\CN_n)_\BQ$ will be called  {\em monomial} Hecke operators.   
 
 Conjecture \ref{conjcom} implies that this homomorphism factors through a homomorphism of the Hecke algebra,
 \begin{equation}
\BT\colon \CH_K\to\End(K^\sigma(\CN_n)_\BQ ).
\end{equation}
\begin{remark}Let $n>3$. Then in general 
\[
(\CT_{\leq t_1}\circ\CT_{\leq t_2}\circ\ldots\circ\CT_{\leq t_r})_*\neq (\CT_{\leq t_1})_*\circ (\CT_{\leq t_2})_*\circ\ldots\circ (\CT_{\leq t_r})_* .
\]
In other words, the composed correspondence (defined by iterated fiber products) does not induce the composition of induced correspondences. This is why we do not define  $\BT_{\varphi_1*\varphi_2*\cdots*\varphi_r}$ as the map induced by the composed geometric correspondence $\CT^{\leq t_1}\circ\CT^{\leq t_2}\circ \cdots\circ\CT^{\leq t_r}$.   To see the difficulty, consider  the $r$-fold iterate $\BT_t^r$ for some $t>0$ and for variable $r$. By \cite[Thm. 1.2]{GHR}, the fiber dimension of the map $\CN_n^{[0,t]}\to \CN_n^{[0]}$ is positive, and hence so is the fiber dimension $d$ of $\CT_n^{\leq t}\to \CN_n^{[0]}$. But then the fiber dimension of  the $r$-fold composition of geometric correspondences $\CT^{\leq t}\circ\CT^{\leq t}\circ \cdots\circ\CT^{\leq t}\to \CN_n^{[0]}$ is equal to $rd$ and hence, for $rd>n$, has irreducible components contained in the special fiber which have  dimension larger than the dimension of the generic fiber. Hence the resulting formal scheme is not flat over $O_{\breve F}$  and the induced map on K-groups is not equal to the $r$-fold composition of the endomorphism $(\CT^{\leq t})_*$.  

This problem would disappear, if we had defined the Hecke correspondences as derived formal schemes. Indeed, the map  $\BT_{\varphi_1*\varphi_2*\cdots*\varphi_r}$ is induced by the composition of derived geometric correspondences $\CT_{\rm der}^{\leq t_1}\circ\CT_{\rm der}^{\leq t_2}\circ \cdots\circ\CT_{\rm der}^{\leq t_r}$, cf. Remark \ref{dercomp}. As the argument above shows, this composition of derived geometric correspondences is not in general a classical formal scheme in our case.  Note that, even if we had defined our Hecke operators in terms of derived formal schemes, this does not seem to help in proving Conjecture \ref{conjcom}. Indeed, the only potential argument we see to prove this commutativity is to relate our Hecke operators to classical Hecke operators in the generic fiber, where the desired commutativity holds, and use some density argument to extend the commutativity  integrally. However, derived schemes seem unsuitable for such density arguments.
\end{remark}
\begin{remark}
It is conceivable that $(\CT^{\leq t})_*=\BT_t$, even when the fiber dimension is positive, i.e., when the maps from $\CT^{\leq t}$ to $\CN_n$ are not flat. Even so, we prefer to write $\BT_t$ as a product of intertwining Hecke operators. Our definition is tentative and only a proof of the AFL for the full Hecke algebra in higher dimension will decide which definition is ``the right one''.
\end{remark}
\begin{remark}
 The argument above uses the fiber dimension of the natural maps $\pi\colon \CN_n^{[t, 0]}\to \CN_n^{[0]}$, resp. $\pi'\colon \CN_n^{[0, t]}\to \CN_n^{[0]}$. In  \cite{GHR}, 
the natural maps $\pi\colon \CN_n^{[t', t]}\to \CN_n^{[t']}$ for arbitrary $t'\neq t$ are considered   (in loc. cit. arbitrary affine Deligne-Lusztig varieties are considered at the point set level). More precisely, consider the   fiber of $\pi$ over an arbitrary point in $\CN_n^{[t']}(\bar k)$ (this fiber  is the set of $\bar k$-points of a projective scheme over $\bar k$). Then, for $n>2$, such a  fiber always has strictly positive dimension, unless $t=0$ and  $t'=2$, in which case the fibers are finite, cf. \cite[Thm. 1.2]{GHR}. 
Hence the morphism $\pi$   is non-flat in all cases outside the case $(t=0, t'=2)$. Indeed, both morphisms $\pi$ and $\pi'$ are relatively representable by  projective morphisms which are finite in the generic fiber. If flatness held, then all fibers in the special fiber would be finite as well, a contradiction. 
 \end{remark}
 
\begin{example}\label{ex:dim2}Let  $n=2$. In this case,   $\CN_2$ can be identified with the Lubin-Tate space $\CM_2$ for $n=2$ via the Serre construction, cf. \cite{Kudla2011}. Here $\CM_2$ parametrizes triples $(Y, \rho)$, where $Y$ is a strict formal $O_{F_0}$-module of dimension one and height 2, and where $\rho$ is a framing with a fixed Lubin-Tate module $\BY$ over $\bar k$. Also, there is a natural identification of $\CN_2^{[2]}$ with $\CN_2$. Indeed, let  $(Y, \iota_Y,\lambda_Y, \rho_Y)$ be an object of $\CN_2^{[2]}$. Then  the inclusion $\ker \lambda_Y\subset Y[\varpi]$ is an equality and hence $\lambda_Y=\varpi\lambda$ for a unique principal polarization $\lambda$; associating to $(Y, \iota_Y,\lambda_Y, \rho_Y)$ the object  $(Y, \iota_Y,\lambda, \rho_Y)$ of $\CN_2$ defines the desired  isomorphism. Under this identification, 
$\CN_2^{[2,0]}$ is isomorphic to $\CM_{2, \Gamma_0}$, the $\Gamma_0(p)$-level covering of $\CM_2$ (generalized from $\BQ_p$ to arbitrary $F_0$). Here $\CM_{0, \Gamma_0}$ parametrizes isogenies  $\alpha: Y\to Y'$ of degree $q$ with $\ker \alpha\subset Y[\varpi]$ which lift a given isogeny $\BY\to\BY$. Note that $\CM_{2, \Gamma_0}(\bar k)$ consists of a single point. Indeed, any two isogenies $\BY\to\BY$ as above differ by a quasi-isogeny of degree $0$. But such a quasi-isogeny is automatically an automorphism: denoting by $D$ the quaternion division algebra over $F_0$, we have 
\[\Aut(\BY)=O_D^\times=\{x\in D\mid \Nm(x)\in O_{F_0}^\times\}=\{x\in\End^o(\BY)\mid \deg(x)=0\}.
\]

In this case,  all maps to $\CN_2$ in the diagram \eqref{diagCN} are finite and flat, and hence so are the corresponding maps for the iterated correspondences $(\CT^{\leq 2})^r=\CT^{\leq 2}\circ\CT^{\leq 2}\circ \cdots\circ\CT^{\leq 2}$. In this case, Conjecture \ref{conjcom} is empty, we have $(\CT_{\leq 2})_* =(\CN_2^{[2,0]})_*\circ (\CN_2^{[0,2]})_*$ and the Hecke operators are induced by geometric Hecke correspondences, i.e., $\BT_t^r=((\CT^{\leq 2})^r)_*$, cf. Lemma \ref{complem}.
\end{example}

  \begin{remark}\label{remLiMi}
Li--Mihatsch \cite{LM} consider a situation related to the Linear ATC/AFL for Lubin-Tate space. They also define  Hecke operators  by first constructing these for  distinguished generators of the Hecke algebra, and then by composition (instead of maps of K-groups, they consider maps of cycle groups). In their case  both projection maps  to $\CN_n$  are finite and flat, and the same  is true for the composition of these distinguished geometric correspondences. Hence their  compositions of distinguished Hecke operators are induced by compositions of geometric correspondences.  In their case, the analogue of Conjecture \ref{conjcom} follows by a density argument for the Hecke correspondences from the classical definition of Hecke correspondences in  the generic fiber, which implies the commutativity in the generic fiber. 
\end{remark}

 \begin{remark}
 Fix an even $t$ with $0\leq t\leq n$. Since $\CN_n^{[t']}$ is regular for all $t'$, the diagram \eqref{RZtt'} defines a Hecke operator $\BT_{{\bf1}_{K^{[t]}K^{[t, t']}}*{\bf1}_{K^{[t',t]}K^{[t]}}}$ corresponding to the element ${\bf1}_{K^{[t]}K^{[t, t']}}*{\bf1}_{K^{[t', t]}K^{[t]}}$ in the Hecke algebra $\CH_{K^{[t]}}$. When $t>0$, it is not clear what subalgebra these elements generate. It is also not clear what the relations are among these elements. 
 \end{remark}

\subsection{Hecke correspondences for the product}\label{ss:RZproduct}
 Let 
 $$
\CN_{n,n+1}= \n\times_{\Spf O_{\breve F}}\N.
$$ We replace the diagram \eqref{RZtt'} by the following diagram, in which  the top is defined by the fact that the square is cartesian,
\begin{equation*}
\begin{aligned}
\xymatrix{&&\CT_{n, n+1}^{\leq t, \leq t'} \ar[rd]  \ar[ld] & \\ &\CN_n^{[0,t]}\times  \CN_{n+1}^{[0,t']}\ar[rd]  \ar[ld] & &\CN_n^{[t,0]}\times  \CN_{n+1}^{[t', 0]} \ar[rd]  \ar[ld] &\\ \CN_{n, n+1} && \CN_n^{[t]}\times  \CN_{n+1}^{[t']}&&  \CN_{n, n+1}.}
\end{aligned}
\end{equation*}
We use this diagram to define the Hecke operator $\BT_\varphi$ when $\varphi\in \CH_{K^\flat\times K} $ is atomic. In this case $t=0$ or $t'=0$ and hence the lower oblique arrows are the identity in one factor; we define $\BT_\varphi=\BT_\varphi^-\circ\BT_\varphi^+$ as in \eqref{elhec}. More precisely, for  closed formal subschemes $A$ of $\CN_{n, n+1}$, resp. $B$ of $\CN_n^{[t]}\times  \CN_{n+1}^{[t']}$, we obtain  
\begin{equation}
\begin{aligned}
   &\BT^{A, +}_\varphi\colon K^A(\CN_{n, n+1}) \to K^{(\CN_n^{[0,t]}\times  \CN_{n+1}^{[0,t']})(A)}( \CN_n^{[t]}\times  \CN_{n+1}^{[t']}), \\
   &\BT^{B, -}_\varphi\colon  K^B( \CN_n^{[t]}\times  \CN_{n+1}^{[t']})\to K^{(\CN_n^{[t, 0]}\times  \CN_{n+1}^{[t', 0]})(B)}(\CN_{n, n+1}) ,\\
    &\BT^A_\varphi\colon K^A(\CN_{n, n+1})\to K^{\CT^{\leq t, \leq t'}(A)}( \CN_{n, n+1}) .
\end{aligned}
\end{equation} 
After this, assuming the obvious analogue of Conjecture \ref{conjcom} for $\CN_{n, n+1}$, we define $\BT^A_\varphi$ for  monomial elements by iterated compositions, and for  general elements $\varphi\in \CH_{K^\flat\times K} $ as linear combinations.

\section{AFL Conjectures for the Hecke algebra}

 In this section, we formulate the AFL conjecture in its homogeneous version and in its inhomogeneous version.  We  consider  the arithmetic diagonal cycle $\Delta\subseteq \CN_{n,n+1}$.  
\subsection{Statement of the AFL for Hecke correspondences (homogeneous version)}\label{ss:hom} Let  $W_0$ be a split hermitian space of dimension $n+1$ and $W_1$ its nonsplit form. Also, let $u_i\in W_i$ of unit norm for $i=0, 1$. We identify $W_1$ with $\mathbb{V}_{n+1}$ defined in \S\ref{sec:herm-space-mathbbv} in such a way that $u_1\in W_1$ is mapped to the element $u_0\in\BV_{n+1}$ which corresponds to the map $\bar\BE\to \BX$ defined by the product decomposition $\BX=\BX^\flat\times\bar\BE$.  Then we may identify $\U(W_1^\flat)$ (resp. $\U(W_1)$) with $\U(\mathbb{V}_{n})$ (resp. $\U(\mathbb{V}_{n+1})$). Then $G_{W_1}(F_0)$ acts on $\n\times\N$  via this identification.

For each  monomial element $\varphi\in  \hkf\otimes\hk$ (cf. Definition \ref{def:basic fun}), we consider  the corresponding geometric correspondence $\CT_\varphi$ with its two projections $\pi_1, \pi_2\colon \CT_\varphi\to \CN_{n, n+1}$ (recall that in general the Hecke operator corresponding to $\varphi$ is not induced by $\CT_\varphi$). Note that $\CT_\varphi$ depends on the order of the product decomposition of $\varphi$ into atomic elements. Let $g\in G_{W_1}(F_0)$, and consider the image $g\Delta$ under the induced automorphism of  $\CN_{n, n+1}$. 
\begin{proposition}
\label{prop:finite}Let  $\varphi\in \hkf\otimes\hk$ be a monomial element. Let $g\in G_{W_1}(F_0)_\rs$ be regular semisimple. Then
the intersection  $$ g\Delta\cap  \CT_\varphi( \Delta)=\pi_2\big(\pi_2^{-1}(g\Delta)\times_{\CT_\varphi} \pi_1^{-1}(\Delta)\big)$$ is a proper scheme, i.e., its ideal of definition contains an ideal of definition of $\CN_{n, n+1}$,  and the underlying reduced scheme is  proper over $\Spec \bar k$.

\end{proposition}

\begin{proof} The proof is based on the same idea as that of  Lemma \ref{lem:orb finite}.  It suffices to consider elements of the form $(1,g)$ with $g\in \RU(W_1)_\rs$. Introduce the locus $\CT_N$ of  points $\big((X^\flat, X), (X^{\prime \flat}, X')\big)$ in $ \CN_{n, n+1}\times \CN_{n, n+1}$   such that the universal quasi-isogenies $f^\flat:X^\flat\to X^{\prime \flat}$ and $f:X\to X'$ defined by the framings, are isogenies after being multiplied by $\varpi^N$  (here we have suppressed all auxiliary structures from the notation). It is well-known that  $\T_N$ is a closed formal subscheme, cf. \cite[\S 2]{RZ96}. Then there exists $N$ such that $\im(\CT_\varphi)\subset \CT_N$. It suffices to prove that $(\Delta\times (1\times g)\Delta)\cap\CT_N$  is a proper scheme. 

The two projection maps from  $\CT_N$ to $\CN_{n, n+1}$ are both proper. Therefore it suffices to show that the image of the intersection under the first projection map is a proper scheme. Let $(X^\flat,X)$, resp.  $(X'^\flat,X')$, be  points in $\Delta$ such that $\big((X^\flat,X), (X'^\flat,g X')\big)$ lies in $(\Delta\times (1\times g)\Delta)\cap\CT_N$. Then 
$X=X^\flat\times\CE$ and $X'=X'^\flat\times \CE$, and the natural quasi-isogenies $f^\flat_1:X^\flat\to X'^\flat$ and $f_2:X\to g X'$ have the property that  $\varpi^N f_1^\flat, \varpi^N f_2$ are isogenies. Consider the quasi-isogeny $f_1=(f_1^\flat,\id_{\CE}): X\to X'$. Then $\varpi^N f_1$ is an isogeny.

Recall the element $u_0\in \BV$ corresponding to $u_1\in W_1$, i.e., the natural element corresponding to the product decomposition $\BX=\BX^\flat\times\bar\BE$. Let $\CZ(u_0)$ be the associated special divisor on $\CN_{n+1}$, cf. subsection \ref{sec:herm-space-mathbbv}.  From $X'\in \CZ(u_0)$ and the isogeny $\varpi^N f_2$, it follows that $X\in \CZ(\varpi^N gu_0)$.  From the isogeny  $\varpi^N f_1$ it follows that  $X'\in \CZ(\varpi^{2N} gu_0)$. Then inductively we see that $X$ lies on the intersection of the special divisors $\CZ(u_0), \CZ(\varpi^N gu_0), \CZ(\varpi^{3N} g^2u_0),\cdots$.  In particular, $X$ lies on the special cycle 
$$
\CZ(\varpi^{(2n-1)N}\pair{ u_0, gu_0, \ldots, g^n u_0}).$$
It is known that this special cycle is a proper scheme, cf. \cite[Proof of Lem. 6.1]{Mihatsch2016}.  This shows that the image $(X^\flat,X)$ of the intersection under the first projection map is a proper scheme. The proof is complete.
\end{proof}

We introduce the \emph{intersection number}  
\begin{equation}\label{defintno}
\Int(g,\varphi):=\langle g\Delta, \BT^\Delta_\varphi(\Delta) \rangle_{\CN_{n,n+1}}, \quad g\in G_{W_1}(F_0)_\rs, \quad \varphi\in \CH_{K^\flat\times K} .
\end{equation}
 Here the  RHS is defined by linearity from the case of monomial $\varphi$. For monomial $\varphi$, it is defined as follows.  Consider the cup product
 \[
 K^{g\Delta}(\CN_{n, n+1})\times K^{\CT_\varphi(\Delta)}(\CN_{n, n+1})\to K^{g\Delta\cap \CT_\varphi(\Delta)}(\CN_{n, n+1}) .
 \]
 On the other hand, there is  the composition of maps
 \[
  K^{g\Delta\cap \CT_\varphi(\Delta)}(\CN_{n, n+1})\xrightarrow{{\rm nat}} K'(g\Delta\cap \CT_\varphi(\Delta))\xrightarrow{\chi}\BZ .
 \]
Here the first map  is the natural map from $K$-theory to $K'$-theory (sending a complex $C$ of locally free sheaves on a formal scheme $X$ whose cohomology sheaves have \emph{formal support} in a formal closed subset $Y$ of $X$ (i.e.,  which are annihilated by some power of the ideal sheaf of $Y$)  to the alternating sum of the classes in $K'(Y)$ of the cohomology sheaves of $C$); the second map is given by the Euler-Poincar\'e characteristic (defined since $g\Delta\cap \CT_\varphi(\Delta)$ is a proper scheme).

 Combining these two maps defines the pairing
\begin{equation*}
 \langle \, , \, \rangle_{\CN_{n, n+1}}\colon K^{g\Delta}(\CN_{n, n+1})\times K^{\CT_\varphi(\Delta)}(\CN_{n, n+1})\to\BZ .
\end{equation*} 
Tensoring with $\BQ$, we obtain the pairing
\begin{equation}\label{cupno}
 \langle \, , \, \rangle_{\CN_{n, n+1}}\colon K^{g\Delta}(\CN_{n, n+1})_\BQ\times K^{\CT_\varphi(\Delta)}(\CN_{n, n+1})_\BQ\to\BQ .
\end{equation}

Consider the  element $[g\Delta]\in K^{g\Delta}(\CN_{n, n+1})$, namely the structure sheaf $\CO_{g\Delta}$, considered as  an element in $K'(g\Delta)=K^{g\Delta}(\CN_{n, n+1})$. Similarly,  consider the element $[\Delta]\in K^{\Delta}(\CN_{n, n+1})$ and its image $\BT^\Delta_\varphi([\Delta])\in K^{\CT_\varphi(\Delta)}(\CN_{n, n+1})_\BQ$ under the map $\BT^\Delta_\varphi\colon K^{\Delta}(\CN_{n, n+1})\to K^{\CT_\varphi(\Delta)}(\CN_{n, n+1})_\BQ$. Then  the RHS of \eqref{defintno} is defined as 
$\langle [g\Delta], \BT^\Delta_\varphi([\Delta]) \rangle_{\CN_{n, n+1}}$. 
 \begin{remark}\label{rmmult}
 Let $\varphi$ be a monomial element.  If Conjecture \ref{conjcom} holds true, then the quantity $\Int(g, \varphi)$ does not depend on the order in which the monomial element $\varphi$ is written as a product of atomic elements. This independence is all that matters in the formulation of the AFL conjecture. Then $\Int(g, \varphi)$ defines a linear functional $\CH_{K^\flat\times K}\to \BQ$. 
 \end{remark}
 
\begin{remark}\label{rem:int inv}
Let $h\in\U(W_1^\flat)(F_0)\times \U(W_1)(F_0)$.  Denote by
 $$h_*\colon K^{\Delta}(\CN_{n, n+1})\to K^{h(\Delta)}(\CN_{n, n+1}), \text{ resp. }\, h_*\colon K^{\BT_\varphi(\Delta)}(\CN_{n, n+1})\to K^{h(\BT_\varphi(\Delta))}(\CN_{n, n+1})$$
   its action on the K-group. Then $h_*\circ\BT^\Delta_\varphi=\BT^{h(\Delta)}_\varphi \circ h_*$.  If $h$ is  a diagonal element induced by an element in $ \U(W_1^\flat)(F_0)$, then $h\Delta=\Delta$. It  follows that for $h_1, h_2\in\U(W_1^\flat)(F_0)$,
\begin{equation*}
\begin{aligned}
\Int(g,\varphi)&=\langle g\Delta, \BT^\Delta_\varphi(\Delta) \rangle=\langle gh_1\Delta, \BT^\Delta_\varphi(h_2\Delta) \rangle\\
&=\langle gh_1\Delta, h_2\BT^\Delta_\varphi(\Delta) \rangle\\&=\langle h_2^{-1}gh_1\Delta, \BT^\Delta_\varphi(\Delta) \rangle\\
&=\Int(h_2^{-1}gh_1,\varphi) .
\end{aligned}
\end{equation*}
Therefore the function $g\in G_{W_1}(F_0)_\rs\mapsto \Int(g,\varphi)$ depends only on the orbit of $g$. 
\end{remark}

\begin{conjecture}\label{AFL}(AFL for the spherical Hecke algebra, homogeneous version.)
Let $\fp'\in \Hkf \otimes \Hk$, and let $\fp={\rm BC}(\fp')\in  \hkf \otimes \hk$. Then
  \begin{equation*}
  2 \Int(g,\fp)\cdot\log q= -\omega (\gamma)\del\big(\gamma,  \fp'\big),
\end{equation*}
whenever $\gamma\in G'(F_0)_\rs$ is matched with  $g\in G_{W_1}(F_0)_\rs$. 
\end{conjecture}

In \S\ref{s:n=1} we prove the conjecture when $n=1$, cf.  Theorem \ref{thm: n=1}.
\begin{remark}\label{explain}
To ensure that the LHS is well-defined, we are implicitly using the commutativity conjecture,  Conjecture \ref{conjcom} (for the product $\CN_{n, n+1}$, cf. \S \ref{ss:RZproduct}) or its weakened version, cf. Remark \ref{rmmult}. However, one may bypass this by interpreting Conjecture \ref{AFL} as saying that the AFL identity holds for any representative $\wt\fp\in\wt\CH_K$ of $\fp$ in the noncommutative polynomial algebra corresponding to the polynomial basis given by the atomic generators, cf. \eqref{wthomo}. 
\end{remark}
\subsection{The inhomogeneous version of the AFL}

This is the special case when  $\varphi={\bf 1}_{K^\flat}\otimes f$.  Moreover we choose $\fp'={\bf 1}_{K'^\flat}\otimes f'$ for $f'\in \CH_{K'}$ with $\Bc_{n+1}(f')=f$. In this case it is more elegant to formulate the analytic side in terms of the symmetric space $S_{n+1}$. In fact, by Lemma \ref{lem:hom2in} we have   for $\gamma=(\gamma_1,\gamma_2)\in G'(F_0)_\rs$,
$$
\omega_{G'}(\gamma) \del\bigl(\gamma,{\bf 1}_{K'^\flat}\otimes f'\bigr)=2\omega_{S}(r(\gamma_1^{-1}\gamma_2))   \del(r(\gamma_1^{-1}\gamma_2),r_{\ast}^{\eta^n}(f')).
 $$ By Remark \ref{rem:int inv}, it suffices to consider the regular semi-simple elements of the form $(1, g)$, with $g\in\U(W_1)(F_0)_\rs$.
 
Recall
from \eqref{BCforS} and  \eqref{BCforS eta} the isomorphism
  \begin{equation*}
 \Bc^{\eta^n}_{S_{n+1}}: \CH_{K'_S}\isoarrow \CH_{K}.
  \end{equation*}
Then the inhomogeneous version AFL is as follows. We emphasize that this conjecture is only a special case of Conjecture \ref{AFL}.
\begin{conjecture}\label{AFL inhom}(AFL for the spherical Hecke algebra $\CH_K$, inhomogeneous version.)
Let $f\in\hk$. Then
  \begin{equation*}
 \Int((1,g), {\bf 1}_{K^\flat}\otimes f)\cdot\log q= -\omega (\gamma)\del\big(\gamma,  (\Bc^{\eta^n}_{S_{n+1}})^{-1}(f)\big),
\end{equation*}
whenever $g\in  \U(W_1)_{\rs}$ is matched with  $\gamma\in S_{n+1}(F_0)_\rs$. 
\end{conjecture}

\begin{remark}
There is also the special case of the AFL conjecture \ref{AFL} when the second factor of $\varphi\in \hkf\otimes\hk$ is the unit element. It is not clear to us how to simplify the analytic side in this case. 
\end{remark}

\section{The case $n=1$}
\label{s:n=1}

In this section we provide evidence to Conjecture \ref{AFL} by proving the case $n=1$. We also give a direct (local) proof of the FL for the full spherical Hecke algebra  in this case (cf. \S\ref{ss:FL}). We use the following notation.  Let $G'=\GL_2(F)$ and $K'=\GL_2(O_F)$. We write $\varpi^{(m, m')}$ for the diagonal matrix in $G'$ with entries $\varpi^{m}$ and $\varpi^{m'}$. Also, $G=\U(W_0)$, where we use the  Hermitian form on the standard $2$-dimensional vector space given by the Hermitian matrix 
\begin{equation}\label{hermat} \begin{pmatrix}
  & \sqrt{\epsilon}\\
-\sqrt{\epsilon} & 
\end{pmatrix}, \quad \epsilon\in O^\times_{F_0}\setminus O^{\times, 2}_{F_0} .
\end{equation}   We let $K=G\cap \M_2(O_F)$ be the natural hyperspecial maximal compact subgroup.

\subsection{Base change homomorphism} 

In this subsection we explicate the base change homomorphism $$\Bc:\Hk \to \CH_K,$$
as well as an auxiliary map and its factorization \eqref{eq:diag1}, which we record here for later use,
 \begin{equation}\label{eq:diag2}
 \begin{aligned}
\xymatrix{ \CH_{K'} \ar[d]_-{r^{\eta}_\ast} \ar[drr]^-{\Bc }& & \\
 \CH_{K'_{S}}\ar[rr]^{\sim}_-{ \Bc^{\eta}_{S}}&
& \CH_K} 
\end{aligned}
\end{equation}

Our goal is to compute the isomorphism $ \Bc^\eta_{S}$ explicitly in terms a certain basis of $\CH_K$. 
 We will use freely statements from \S\ref{s:FL}.

Set $$
 f_m'={\bf 1}_{ M_2(O_F)_{v\circ \det=m}},\quad m\geq 0.
 $$
 The
 Satake isomorphism  $\Hk\simeq \BC[X,Y]^{W}$ for $\GL_2(F)$
 is explicitly given by 
 \begin{equation}\label{Sat f'm}
 \Sat( f'_m)=q^{m}\frac{X^{m+1}-Y^{m+1}}{X-Y} .
 \end{equation}Here and below we use $X=x_1, Y=x_2$ for $x_1,x_2$ in \S\ref{s:FL}.
 
  Similarly, setting 
 $$
 f_m={\bf 1}_{ \varpi^{-m}M_2(O_F)\cap G},
 $$
  the
 Satake isomorphism  for the unitary group is  explicitly given 
   by 
   \begin{equation}\label{Sat fm}
 \Sat(  f_m)=q^{m}\frac{X^{(2m+1)/2}-Y^{(2m+1)/2}}{X^{1/2}-Y^{1/2}} ,
 \end{equation}
 where $XY=1$. Note that, despite the fractional exponents, the last expression is a polynomial of $X$ and $Y=X^{-1}$:
  \begin{equation}\label{Sat fm alt}
 \Sat(  f_m)=q^{m}\sum_{i=-m}^m X^{i}. 
 \end{equation}
 
   Set 
\begin{equation}\label{defprime}
\phi_m:={\bf 1}_{K\varpi^{(m,-m)} K}=f_m-f_{m-1}.
\end{equation}
Here, and the sequel, we set $f_m'=0$ for $m<0$, and similarly for $f_m$.   Then we have a basis of $\CH_K$ given by $\phi_m, m\geq 0$. In terms of the functions introduced in \S \ref{s:FL} and \S \ref{s:altbas}, we have $\phi_1=f^{[2]}$, cf. \eqref{def:f}, hence $\phi_1=\varphi_2-(q+1)$, cf. \eqref{ftophi}.

The  base change homomorphism is then the natural quotient map defined by setting $XY=1$:
\begin{align}\label{eq:Bc}
\Bc\colon \BC[X,Y]^{W}\to \BC[X, X^{-1}]^{W}\simeq \BC[X+X^{-1}].
\end{align}

Next we describe explicitly the other two maps in the diagram \eqref{eq:diag2}.  Consider the Cartan decomposition for the symmetric space $S_2$, 
$$
S_2(F_0)=\coprod_{m\geq 0} \K\cdot  \begin{pmatrix}
  & \varpi^{m}\\
 \varpi^{-m} & 
\end{pmatrix}.
$$
The subset indexed by $m=0$ is  $K'_S=K'\cdot 1_2=K'\cap S_2(F_0)$.
 We denote the complete set of representatives of $\K$-orbits on $S_2(F_0)$,
\begin{equation}\label{eq:t m}
 t_m:= \begin{pmatrix}
  & \varpi^m\\
 \varpi^{-m} & 
\end{pmatrix},\quad m\geq 0. 
\end{equation}
Therefore, we have a basis of $\CH_{K'_S}$  given by 
\begin{equation}\label{eq: varphi' m}
\varphi'_m={\bf 1}_{\K \cdot  t_m},\quad m\geq 0.
\end{equation}

 Recall the map $r^\eta_*$ from \eqref{eq:r eta}. Note that a basis of $\Hk$ is given by $ {\bf 1}_{\varpi^j \K\varpi^{(m,0)}\K}, j\in\BZ,m\geq 0 $. 
\begin{lemma}\label{lem:BC m}
\begin{altitemize}
\item[(i)]
The map
 $r^\eta_*\colon\Hk \to \CH_{K'_S}$ sends ${\bf 1}_{\varpi^j \K\varpi^{(m,0)}\K}$ to  $(-1)^m\sum_{i=0}^m e_{m-i} \varphi'_{i}$, where 
$$
e_i=\begin{cases} 1, &i=0,\\
q^{i}(1+q^{-1}), & i>0.
\end{cases}
$$
\item[(ii)]
The map $r^\eta_*$
 sends $f'_m$ to $(-1)^m\sum_{i=0}^m (\sum_{j=0}^{m-i} q^j ) \varphi'_{i}.$
\item[(iii)]We have  $
\Bc^\eta_S(\tilde\varphi'_{m})=\phi_{m}$,
where \begin{align}\label{wt fp'}
\tilde\varphi'_{m}:= (-1)^m(  \varphi'_m+2\varphi'_{m-1}+\cdots+2 \varphi'_0)\in \CH_{K'_S}.
\end{align}

\end{altitemize}
\end{lemma}
\begin{proof}
We first show part (ii). We need to compute the integral 
$$r^\eta_*(f'_m)(g\ov g^{-1})=\int_{\GL_2(F_0)} f'_m(gh)\wt\eta(gh)\, dh.
$$Since the determinant of any element in the support of $f'_m$ has valuation $m$, the integral is equal to
$$r^\eta_*(f'_m)(g\ov g^{-1})=(-1)^m\int_{\GL_2(F_0)} f'_m(gh)\, dh.
$$
It suffices to determine its value at elements of the form $t_\ell$, cf. \eqref{eq:t m}. By Iwasawa decomposition, every element in $\GL_2(F)$ lies in
 in some $\K \begin{pmatrix}
1& u\\
 & 1 
\end{pmatrix} \begin{pmatrix}
   \varpi^{m-i}&\\
 & \varpi^i 
\end{pmatrix}
$. All elements in this $K'$-coset are mapped into a single $\K$-orbit  in $S_2(F_0)$ with representative 
$$ \begin{pmatrix}
1& u\\
 & 1 
\end{pmatrix} \begin{pmatrix}
1& \ov u\\
 & 1 
\end{pmatrix}^{-1} = \begin{pmatrix}
1& u-\ov u\\
 & 1 
\end{pmatrix} .$$
This last element is $\K$-equivalent (in $S_2(F_0)$) to $t_{\max\{-v(u-\ov u),0\}}$ (recall that we are assuming that the residue characteristic is odd).  Therefore it suffices to compute 
the  integral for $g=\begin{pmatrix}
1&  u\\
 & 1 
\end{pmatrix}$, where $u\in F^{-}$ lies in the purely imaginary part and has valuation $v(u)\leq 0$.

We use the Iwasawa decomposition $\GL_2(F_0)=ANK_0$, where  $A$ denotes the diagonal torus, $N$ the subgroup of upper triangular unipotent matrices, and  $K_0=\GL_2(O_{F_0})$.  Accordingly we write an element in $\GL_2(F_0)$ as $h=\begin{pmatrix}
x& \\
 & y 
\end{pmatrix}\begin{pmatrix}
1& z \\
 & 1 
\end{pmatrix} k$.  The Haar measure on $\GL_2(F_0)$ is then given by $d^\times x \, d^\times y\,dz$, where the multiplicative (resp. additive) Haar measure on $F_0^\times$ (resp. $F_0$) is normalized such that $\vol(O_{F_0}^\times)=1$ (resp. $\vol(O_{F_0})=1$). Then the condition $gh\in M_2(O_F), v(\det(gh))=m$ is equivalent to 
$$
x,y\in O_{F_0}, \quad v(xy)=m,\quad  xz\in O_{F_0},\quad  y u\in O_{F} 
$$
(note that $x,y,z\in F_0$ and $u\in F^{-}$). It follows that 
$$
\int_{\GL_2(F_0)} f'_m(\begin{pmatrix}
1&  u\\
 & 1 
\end{pmatrix} h)\, dh= \int_{-v(u)\leq v(y)\leq m} \int_{v(x)=m-v(y)} \int_{z\in \frac{1}{x}O_{F_0}} dz \,d^\times x \, d^\times y.
$$
This triple integral is equal to 
$$
\int_{\GL_2(F_0)} f'_m(\begin{pmatrix}
1&  u\\
 & 1 
\end{pmatrix} h)\, dh=\sum_{0\leq i\leq m+v(u)} q^i.
$$
We thus obtain the value of $r^\eta_*(f_m')$ at the element $t_\ell$ in \eqref{eq:t m},
$$
r^\eta_*(f_m')(t_{\ell})=(-1)^m\begin{cases}\sum_{0\leq i\leq m-\ell} q^i,&0\leq \ell\leq m, \\
0, &\ell>m.
\end{cases}
$$
A comparison with \eqref{eq: varphi' m} yields
$$
r^\eta_*(f_m')=(-1)^m \sum_{i=0}^m (\sum_{j=0}^{m-i} q^j ) \varphi'_{i}.
$$
This shows part (ii). To show part (i),
by ${\bf 1}_{K' \varpi^{(m,0)}K'} =f'_m-f'_{m-2}$, we deduce from part (ii) that  
$$r^\eta_*({\bf 1}_{K' \varpi^{(m,0)}K'})=r^\eta_*( f'_m-f'_{m-2})=\sum_{i=0}^m e_{m-i} \varphi'_{i}.$$

It remains to show part (iii). From the explicit formulas of Satake transforms  \eqref{Sat f'm} and \eqref{Sat fm}, we see that the map $\Bc$ in \eqref{eq:Bc}
sends $f_m'+q f'_{m-1}$ to $ f_m$. It follows that 
\begin{equation}\label{Bc=phi}
\Bc(f_m'+(q-1) f'_{m-1}-f'_{m-2} )=\phi_m,\quad m\geq 0.
\end{equation}
From part (ii) we have
$$r^\eta_\ast(f_m'+(q-1) f'_{m-1}-f'_{m-2})=\tilde\varphi'_{m}.$$
 By the commutative diagram  \eqref{eq:diag2}, we see that $\Bc_S^\eta$ sends $\tilde\varphi'_{m}$ to $\phi_m$, as desired.
\end{proof}

\begin{remark}
The explicit description above gives a direct proof of \eqref{BCforS} in the case $n=2$.
\end{remark}
 
\subsection{Orbital integrals on the unitary group} 

We now consider the orbital integrals of elements in $\CH_{K}$. Recall from \eqref{defprime} that $\phi_{m}$ denotes $ {\bf 1} _{K \varpi^{(m, -m)}K}$, where, we recall,  $K$ denotes the hyperspecial subgroup of $U(W_0)$.
Now we make a change of Hermitian form for $W_0$ from \eqref{hermat} to the identity matrix. This is for the convenience of orbit comparison in \eqref{eq:orb inhom}. We define the invariants map on the orbits of $G(F_0)$ by  
\begin{equation}\label{invun}
g= \begin{pmatrix}
 a & b\\c
& d
\end{pmatrix}\mapsto  (a,d, bc) ,
\end{equation}
 (note that $(a, d, bc)$ are not independent).

\begin{proposition}\label{prop:orb U}Let $m\geq 1$. Then  for
$g= \begin{pmatrix}
 a & b\\c
& d
\end{pmatrix}\in G(F_0)$  regular semisimple, we have
\begin{align}\label{eq:Orb U}
\Orb(g, \phi_m)= \begin{cases} 1, & v(1-a\ov a)=-2m\\
 0, & v(1-a\ov a)\neq -2m.
 \end{cases}
\end{align}
\end{proposition}
\begin{proof} We have,  since we  use the  Hermitian form on $W_0$ given by  the identity matrix,  
\begin{align}
a\bar a=d\bar d=1-b\bar b, \quad a\bar c=b\bar d\neq 0.
\end{align}
 The invariants  are $a,d$ and $bc=(1-a\ov a)d/\ov a$. Since the group $H=\U_1(F_0)$ is compact and lies in $K$, the orbital integral
is either $1$ or $0$, depending on whether $g$ lies in the support of $\phi_m$ or not. We note that the support of $\phi_m$ is the set of matrices $g\in G(F_0)$ such that $\varpi^m g\in M_2(O_F)$ and  such that there exists at least one entry of $g$ with valuation exactly $-m$.  With the help of the above identities, it is easy to see that the support condition amounts to  $a,b,c,d$ all having valuation equal to $-m$.

\end{proof}

\subsection{Orbital integrals on the symmetric space $S_2$} 
Recall that the inhomogeneous orbital integral is defined by  \eqref{eqn def inhom}. 
We also recall from \cite[\S15.1]{RSZ2} the structure of regular semisimple sets 
on the symmetric space $S(F_0) = S_2(F_0)$.  We write an element as
\begin{equation}\label{gammacoo}
\gamma=
\begin{pmatrix}
		a &  b \\
         c  &  d
      \end{pmatrix}\in S(F_0).
\end{equation}
Then $\gamma$ is regular semi-simple if and only if $bc\neq 0$, in which case we may write $\gamma$ as
\begin{equation}\label{gamma(a,b)}
\begin{aligned}
\gamma=\gamma(a,b) &:=\begin{pmatrix}
		a &  b \\
         (1-\RN a)/\ov b &  -\ov ab/\ov b
      \end{pmatrix}\\
		&\phantom{:}=\begin{pmatrix}
		1&   \\
          &  -b/\ov b
      \end{pmatrix} \begin{pmatrix}
		a &  b \\
        - (1-\RN a)/ b &  \ov a
      \end{pmatrix} \in S(F_0)_\rs,
		\quad a \in F \smallsetminus F^1,\ b \in F^\times.
\end{aligned}
\end{equation}Similarly to the unitary group case,
we define the invariant map on the  orbits of $S(F_0)$,
\begin{equation} \gamma= \begin{pmatrix}
 a & b\\c
& d
\end{pmatrix}\mapsto (a,d, bc) ,
\end{equation}
 (again, $(a,d, bc)$ are not independent, cf.  \eqref{invun}). Then an orbit of $\gamma\in S(F_0)_\rs$ matches an orbit of $g\in G(F_0)_\rs$ if and only if they have the same invariants.  In particular, an element $\gamma\in S(F_0)_\rs$ matches  an element in the quasi-split (resp. non-quasi-split) unitary group if and only if $v(1- a\bar a)$ is even (resp. $v(1-a\bar a)$ is odd).

We now consider the orbital integral of elements in $\CH_{K'_S}$. Since $F/F_0$ is unramified, we may assume that, up to conjugation by $H=\GL_1(F_0)$,  for a regular semisimple $\gamma$ as in \eqref{gammacoo} that  the entry $c$ is a unit.
\begin{proposition}\label{prop:orb 1}Let $\gamma=  \begin{pmatrix}
 a & b\\c
& d
\end{pmatrix}\in S_2(F_0)$ be regular semisimple  with $v(c)=0$.
When $m\geq 1$, we have
\begin{align}\label{eq:Orb s}
\Orb(\gamma, \varphi'_m,s)= \begin{cases} (-1)^m q^{ms}(1+\eta(1-a\ov a)q^{-(v(1-a\ov a)+2m)s}), & v(1-a\ov a)> -2m\\
 (-1)^m q^{ms}, & v(1-a\ov a)=-2m\\
 0, & v(1-a\ov a)<-2m .
 \end{cases}
\end{align}
When $m=0$, we have 
\begin{align}\label{eq:Orb s m=0}
\Orb(\gamma, \varphi'_0,s)= \begin{cases}  \sum_{i=0}^{v(1-a\ov a)} (-1)^i q^{-is}, & v(1-a\ov a)\geq 0\\
 0, & v(1-a\ov a)<0 .
 \end{cases}
\end{align}
\end{proposition}

\begin{proof}
Let $\gamma=  \begin{pmatrix}
 a & b\\c
& d
\end{pmatrix}\in S_2(F_0)$ be regular semisimple. Then we have
\begin{align}
1-a\bar a=1-d\bar d= c\bar b=b\bar c\neq 0.
\end{align}

Assume $m>0$.
We first consider those orbits $\gamma$ such that $a\in O_F$. Then $d\in O_F$ and $b\bar c\in O_F$. Then the condition for 
$h^{-1}\gamma h=\begin{pmatrix}
 a & h^{-1}b\\ch
& d
\end{pmatrix}$ to be in $\supp(\varphi'_m)$ is that either $h^{-1}b$ or $ch$ lies in $\varpi^{-m}O_F^\times$. It follows that $v(h)\in\{-m, v(bc)+m\}$ and
\begin{align}
\Orb(\gamma, \varphi'_m,s)= (-1)^m q^{ms}(1+\eta(b\bar c)q^{-(v(b\bar c)+2m)s}).
\end{align}

Next we consider the orbits with $a\notin O_F$. Then  $a\bar a$, $d\bar d$ and $ c\bar b=b\bar c$ all have equal valuations.
If $v(a)=-m$, then $v(b)=2v(a)=-2m$ (note that we have assumed $v(c)=0$). The condition for $h\cdot \gamma\in \K \cdot t_m$ is 
$$
v(h^{-1}b)\geq -m,\quad v(ch)=v(h)\geq -m ,
$$
or equivalently
$$
-m\leq v(h)\leq m+v(b)=-m.
$$
Therefore the orbital integral is equal to
\begin{align}
\Orb(\gamma, \varphi'_m,s)= (-1)^m q^{ms}.
\end{align}
If $-m<v(a)<0$, then $v(b)=2v(a)\neq -2m $. A similar argument shows that $v(h)\in\{-m, m+v(b)\}$ and 
\begin{align}
\Orb(\gamma, \varphi'_m,s)= (-1)^m q^{ms} (1+\eta(b\bar c)q^{-(v(b\bar c)+2m)s}).
\end{align}
We have thus proved \eqref{eq:Orb s}.

Assume $m=0$. If $v(a)\geq 0$, then 
\begin{align}
\Orb(\gamma, \varphi'_0,s)= \sum_{i=0}^{v(bc)} (-1)^i q^{-is}.
\end{align}
The case $v(a)<0$ is similar to the $m>0$ case. 

Finally note that $v(b)=v(bc)=v(1-a\ov a)$. The proof is complete.
\end{proof}

\begin{proposition}\label{prop del O n=1}
\begin{altitemize}
\item[(i)]
 The functions $\phi_{m}\in \CH_K $ and $\tilde\varphi'_{m}\in \CH_{K'_S}$ are transfers of each other.
\item[(ii)]Let $\gamma=  \begin{pmatrix}
 a & b\\c
& d
\end{pmatrix}\in S_2(F_0)_\rs$   with $v(c)=0$. Assume   $v(1-a\ov a)$  odd (so that $\gamma$ matches an element in $\U(W_1)(F_0)_\rs$). When $m\geq 1$, we have 
$$
\del(\gamma,\tilde\varphi'_{m})= \log q \begin{cases}  1, & v(1-a\ov a)> 0\\
 0, & v(1-a\ov a)<0 .
 \end{cases}
$$
When $m=0$, we have
$$
\del(\gamma,\tilde\varphi'_0)= \log q \begin{cases}  \frac{v(1-a\ov a)+1}{2}, & v(1-a\ov a)\geq 0\\
 0, & v(1-a\ov a)<0 .
 \end{cases}
 $$
\end{altitemize}
\end{proposition}
\begin{proof}
When $m>0$, we have the value at $s=0$,
\begin{align}\label{orb fp m'}
\Orb(\gamma, \varphi'_m)=(-1)^m \begin{cases}  1+\eta(1-a\ov a), & v(1-a\ov a)> -2m,\\
 1, & v(1-a\ov a)=-2m,\\
 0, & v(1-a\ov a)<-2m,
 \end{cases}
\end{align} and,
when $ v(1-a\ov a)$ is odd, we have the first derivative at $s=0$,
$$
\del(\gamma,\varphi'_{m})= (-1)^m\log q \begin{cases}  v(1-a\ov a)+2m, & v(1-a\ov a)> -2m\\
 0, & v(1-a\ov a)<-2m .
 \end{cases}
$$

When $m=0$, we have the value at $s=0$,
\begin{align}\label{orb fp 0'}
\Orb(\gamma, \varphi'_0)=(-1)^m \begin{cases}  1, & v(1-a\ov a)\geq 0 \\
 0, & v(1-a\ov a)<0 ,
 \end{cases}
\end{align} 
and, when $ v(1-a\ov a)$ is odd, we have the first derivative at $s=0$,
$$
\del(\gamma,\varphi'_0)= (-1)^m\log q \begin{cases}  \frac{v(1-a\ov a)+1}{2}, & v(1-a\ov a)\geq 0\\
 0, & v(1-a\ov a)<0 .
 \end{cases}
$$

Hence, when  $v(1-a\ov a)=-2m$, we have $\Orb(\gamma, \wt\varphi'_m)=\Orb(\gamma, \varphi'_m)=1$ (cf. the definition of $\wt\fp'_m$ in \eqref{wt fp'}).  When   $v(1-a\ov a)=-2(m-i)>-2m$ is even, we have from \eqref{orb fp m'} and \eqref{orb fp 0'}
\begin{align*}
\Orb(\gamma, \tilde\varphi'_m)=2-4+4+\cdots+(-1)^{i-1}4+(-1)^i2=0,
\end{align*} 
It follows from a comparison with Proposition \ref{prop:orb U} that $\phi_{m}\in \CH(U(V)) $ and $\tilde\varphi'_{m}\in \CH(S_2)$ are transfers of each other. This proves part (i).

It remains to compute the first derivative $\del(\gamma,\wt\varphi'_{m})$ when $v(1-a\ov a)>-2m$ is odd. It suffices to consider the case when  $v(1-a\ov a)=v(bc)$ is positive. Set $v(1-a\ov a)=-2(m-i)+1$.  Then $\del(\gamma,\tilde\varphi'_{m})$  equals $\log q$ times
\begin{align*}
&(v(bc)+2m)-2(v(bc)+2(m-1))+\cdots+(-1)^{m-1}2(v(bc)+2)+(-1)^m(v(bc)+1)\\
=&(2m-2(m-1))-\cdots +(-1)^{i-1}(2(m-i+1)-2(m-i))+\cdots+(-1)^{m-1}(2-1)\\
=&2-2+2-\cdots+(-1)^{m-2}2+(-1)^{m-1}\\
=&1.
\end{align*} 
The proof is complete.
\end{proof}

\begin{remark}
Part (i) gives a direct proof in the case $n=1$ of the FL for the full spherical Hecke algebra, cf. \S\ref{ss:FL} and \S\ref{ss:FL inhom}.
\end{remark}

\subsection{Intersection numbers}
In this subsection, we often write $\CN^{[0]}$, or simply $\CN$, for $\CN_2^{[0]}$. Note that $\CN\simeq \Spf(W[[t]])$ has only one point. Recall the Hecke operator $\BT^A_{\varphi_2}\colon K^A(\CN)\to K^{\CT_{\leq 2}(A)}(\CN)$. It is defined as  $\BT^A_{\varphi_2}=\BT^A_{\leq 2}=\BT^{\CN^{[0,2]}(A),-}\circ \BT^{A, +}$, where 
\begin{equation}\label{T+-}
\begin{aligned}
\BT^{A,+}&=(\CN^{[0,2]})_*\colon K^A(\CN_2^{[0]})\to K^{\CN^{[0,2]}(A)}(\CN_2^{[2]}), \\
\BT^{B,-}&=(\CN^{[2, 0]})_*\colon K^B(\CN_2^{[2]})\to K^{\CN^{[2,0]}(B)}(\CN_2^{[0]}), 
\end{aligned}
\end{equation}cf. \eqref{elhec}. We may identify $\CN_2^{[2]}$ with $\CN$ (cf. Remark \ref{ex:dim2}), and then both Hecke operators in \eqref{T+-} are induced by the same geometric correspondence $\CT_{\Gamma_0}\to\CN\times\CN$ and its transpose, i.e., we can write  
\begin{equation}
\CT_{\leq 2}=\CT_{\Gamma_0}\circ{^t\CT_{\Gamma_0}} .
\end{equation} 
The notation $\CT_{\Gamma_0}$ is a reminder of the fact that $\CT_{\Gamma_0}$ is the analogue of the $\Gamma_0(p)$-covering of the modular curve in the present context, cf. Remark \ref{ex:dim2}. 

Since the projection morphisms from the composition $\CT_{\leq 2}$ of the geometric correspondences $\CN^{[0,2]}$ and $\CN^{[2, 0]}$ to $\CN$ are finite and flat, we have
$\BT_{\leq 2}=(\CT_{\leq 2})_* $, cf. Lemma \ref{complem} (we are leaving out  the support $A$ from the notation).
A similar reasoning shows that $\BT_{\varphi_2^m}=((\CT_{\leq 2})^m)_*$, where $(\CT_{\leq 2})^m=\CT_{\leq 2}\circ \cdots \circ\CT_{\leq 2}$ is the $m$-fold iterated composition of $\CT_{\leq 2}$ with itself (which again maps by finite flat morphisms to $\CN^{[0]}$). 

At this point we use the local intersection calculus developed in \cite[\S 5]{LM} (the role of the regular local formal scheme $M$ in loc.~cit. is played here by $\CN$).

We consider the natural morphism $(\CT_{\leq 2})^m\to \CN\times\CN$. Since $(\CT_{\leq 2})^m$ is finite and flat over $\CN$,  this morphism  is finite with image of  codimension one. We associate to it an element of the group $Z^1(\CN\times\CN)$ of cycles of codimension one (namely $\sum_{Z\in (\CN\times\CN)^{(1)}}\ell_{\CO_{\eta_Z}}(\CO_{(\CT_{\leq 2})^m})[Z]\in Z^1(\CN\times\CN)$,  cf. \cite[Def. 5.1]{LM}).  We use the same notation $(\CT_{\leq 2})^m$ for this one-cycle. Since $\CN\times\CN$ is regular,  we may regard $(\CT_{\leq 2})^m$ as an effective  Cartier divisor on $\CN\times\CN$. By the flatness of $(\CT_{\leq 2})^m$ over $\CN$, it is in fact a relative Cartier divisor, i.e., it is at every point defined by one equation in $W[[t, t']]$ which is neither a unit nor divisible by $p$.   

Both projection maps to $\CN$ are finite, hence  $(\CT_{\leq 2})^m$ is an element of the ring of correspondences ${\rm Corr}(\CN)$, and from the definition of the ring structure on ${\rm Corr}(\CN)$, it follows that $(\CT_{\leq 2})^m$ is the $m$-th power of the correspondence $\CT_{\leq 2}$, cf. \cite[Def. 5.6]{LM}. We obtain a homomorphism of $\BQ$-algebras,
\[
\CH_K\to {\rm Corr}(\CN\times\CN)_\BQ ,\quad \sum_m a_m\varphi_2^m\mapsto \sum_m a_m(\CT_{\leq 2})^m .
\]
 We denote the image of $\varphi$ under this map by $\CT_{\varphi}$. Note that  the unit element $\varphi={\bf 1}$ is mapped to  $\CT_{\bf 1}=\Delta_{\CN\times\CN}$ (the diagonal in $\CN\times\CN$), which we denote by $\CT_0$. To simplify the notation, we  write $\CT_1$ for $\CT_{\varphi_2}$ (not to be confused with the divisor associated to the unit element). 
Also, for $\phi_{m}={\bf 1}_{K \varpi^{(m,-m)} K }$, we use the notation  
$\CT^\circ_{m}$ for $\CT_{\phi_m}$. Hence $\CT^\circ_{0}=\CT_{0}$ is the unit element in ${\rm Corr}(\CN\times\CN)_\BQ$.
\begin{lemma}\label{Cartind} There is the relation of Cartier divisors on $\CN\times\CN$,
$$\CT^\circ_{1}=\CT_1-(q+1)\CT_0 ,$$ and the recursive relation
$$
\CT_{\varphi_2*\phi_m}=\CT_{1}\circ \CT^\circ_{m}=\CT^\circ_{m+1} +2q \CT^\circ_{m}+q^2\CT^\circ_{m-1}, \quad m\geq 1.
$$
\end{lemma} 
\begin{proof}
This follows from the relations in the Hecke algebra:
\[
\phi_1=\varphi_2-(q+1), \quad \varphi_2\phi_m=\phi_{m+1}+2q\phi_m+q^2\phi_{m-1} , m\geq 1.
\]
The first identity is mentioned below \eqref{defprime}. The second identity follows from this, after expressing $\phi_m$ in terms of $f_m$ as in \eqref{defprime} and applying the Satake isomorphism, cf. \eqref{Sat fm alt}. 
\end{proof}
Recall the action of ${\rm Corr}(\CN)$ on $Z^1(\CN)$, cf. \cite[Def. 5.7]{LM}
\begin{equation}
\CZ\mapsto \CT*\CZ .
\end{equation}
Hence we obtain an action of $\CH_K$ on $Z^1(\CN)_\BQ$.

 For a non-zero $u\in \BV_2$ we let $\CZ(u)$ be the associated special divisor and $\CZ(u)^\circ=\CZ(u)-\CZ(\varpi^{-1}u)$  the difference divisor, cf. \cite{Ter}. Both are effective Cartier divisors on $\CN$. Recall the special vector $u_0\in\BV_2$, cf. \S \ref{ss:hom}. It is of valuation zero (i.e., its length is a unit), and $\CZ(u_0)^\circ=\CZ(u_0)$.
\begin{proposition}\label{thm: T Del}
Let $m\geq 0$. Then there is an equality of Cartier divisors on $\CN$,
$$
\CT^\circ_{m}*\CZ(u_0)= \CZ(\varpi^m u_0)^\circ.
$$
(For $m\geq 1$, both sides have degree $q^{2m-1}(q+1)$ over $\Spf O_{\breve{ F}}$.) 

\end{proposition}
\begin{proof}
By Lemma \ref{Cartind}, this identity is equivalent to the conjunction of the following identities. 
\begin{enumerate}
\item $\CT_1*\mathcal{Z}(u_0)=\mathcal{Z}(\varpi u_0)^\circ+(q+1)\mathcal{Z}(u_0).$
\item  $\CT_{1}*\CZ(\varpi^m u_0)^\circ=\CZ(\varpi^{m+1} u_0)^\circ+2q \CZ(\varpi^{m} u_0)^\circ+ q^2\CZ(\varpi^{m-1} u_0)^\circ, \quad m\geq 1.$
\end{enumerate}

We first  recall  the theory of quasi-canonical divisors for $\CN$, cf.  \cite{Kudla2011}. Recall  the framing object $\BX=\BE\times\bar\BE$ (note that in loc.~cit. the order of the factors is reversed, which results in slightly different formulas below). Then $u_0={\rm incl}_2\in \BV_2=\Hom_{O_F}(\bar\BE, \BX)$.  Let $u_1={\rm incl}_1\circ\Pi$, where $\Pi$ is a fixed uniformizer of the quaternion algebra $D=\End(\BE)^o$, comp. Example \ref{ex:dim2}. Then $\mathbb{V}=\langle u_0\rangle_F\obot\langle u_1\rangle_F$ with $\val(u_1)=1$. 

The quasi-canonical divisor of level $s$  is given as follows.  Let $\CF_s/O_{\Fb, s}$ be a quasi-canonical lift of level $s$ of $\BE$ in the sense of Gross. Here for any $n$, $O_{\Fb, n}$ denotes the ring of integers in the abelian extension of $\Fb$ corresponding to the subgroup  $(O_{F_0}+\varpi^n O_F)^\times$ of the group of units.  Let  $X^{(s)}=O_F\otimes_{O_{F_0}}\CF_s$ (Serre construction). Then in \cite[\S 6]{Kudla2011},  $X^{(s)}$ is completed to an object $(X^{(s)}, \iota_{X^{(s)}}, \lambda_{X^{(s)}}, \rho_{X^{(s)}})$ of $\CN(O_{F, s})$, and it is proved that the corresponding morphism $\Spf O_{F, s}\to \CN$ is a closed embedding defining a Cartier divisor $\CZ_s$, the \emph{quasi-canonical divisor} of level $s$. Then $\CZ(u_0)=\CZ_0$ and $\CZ(u_1)=\CZ_1$ and $\CZ(\varpi u_0)=\CZ_0+\CZ_2$, cf. \cite[Prop. 8.1]{Kudla2011}. The last identity can also be written as $\CZ(\varpi u_0)^o=\CZ_2$. More generally, $\CZ(\varpi^m u_0)^o=\CZ_{2m}$ and $\CZ(\varpi^m u_1)^o=\CZ_{2m+1}$, cf. loc.~cit.

Let us now prove (1). Note that  $\mathcal{T}_{\Gamma_0}\cap \pi_1^{-1}(\mathcal{Z}(u_0))\subseteq \mathcal{N}\times\mathcal{N}$ is the locus of  isogenies $X^{(0)}\rightarrow Y$ of degree $q^2$.  Recall the isogeny of quasi-canonical lifts $\psi_1: \CF_0\to\CF_1$, cf. \cite[Lem. 6.2]{Kudla2011} (it depends on the choice of $\Pi$).  By the functoriality of the Serre construction, it defines  the isogeny 
 $X^{(0)}\to X^{(1)}$ over $\CZ(u_1)$. Hence we obtain a closed embedding $\iota\colon\CZ(u_1)\subset \mathcal{T}_{\Gamma_0}\cap \pi_1^{-1}(\mathcal{Z}(u_0))$. Since $\pi_2\circ\iota$ is the natural embedding of $\CZ(u_1)$ in $\CN$, we obtain an inequality of Cartier divisors, $\mathcal{Z}(u_1)\leq [\mathcal{T}_{\Gamma_0}]*\mathcal{Z}(u_0)$. Since both these Cartier divisors  map  with degree $q+1$ to $\Spf O_{\breve F}$, they are equal, i.e., 
\begin{equation}\label{eq1}
\mathcal{T}_{\Gamma_0}*\mathcal{Z}(u_0)=\mathcal{Z}(u_1).
\end{equation}

Similarly, $\mathcal{T}_{\Gamma_0}\cap \pi_1^{-1}(\mathcal{Z}(u_1))$  parametrizes isogenies $X^{(1)}\rightarrow Y$.  The isogeny of quasi-canonical lifts $\psi_{1, 2}: \CF_1\to\CF_2$ (defined in an analogous way to $\psi_{s}: \CF_0\to\CF_s$ for any $s$ in \cite[\S 6]{Kudla2011}), defines the isogeny $X^{(1)}_{\Spf O_{\Fb,2}}\to X^{(2)}$. This defines a closed embedding $\CZ_2=\mathcal{Z}(\varpi u_0)^\circ\subset \CT_{\Gamma_0}*\mathcal{Z}(u_1)$.

 On the other hand, consider the natural isogeny $ X^{(0)}_{\Spf O_{\Fb,1}}\to X^{(1)}$. It induces a unique isogeny $X^{(1)}\to X^{(0)}_{\Spf O_{\Fb,1}}$ such that the composition is equal to $\varpi\colon X^{(1)}\to X^{(1)}$. We obtain  a closed embedding of $\Spf O_{\breve F,1}$ into  $\mathcal{T}_{\Gamma_0}\cap \pi_1^{-1}(\mathcal{Z}(u_1))$, which is mapped under $\pi_2$ onto $\mathcal{Z}(u_0)$,  with degree $[O_{\Fb,1}: O_{\Fb}]=q+1$.  Altogether we obtain an inequality of Cartier divisors, $ \mathcal{Z}(\varpi u_0)^\circ+(q+1)\mathcal{Z}(u_0)\leq \CT_{\Gamma_0}*\mathcal{Z}(u_1)$. Since both divisors have degree $(q+1)^2$ over $\Spf O_{\breve F}$,  we obtain the equality of divisors
\begin{equation}\label{eq2}
\CT_{\Gamma_0}*\mathcal{Z}(u_1)=\mathcal{Z}(\varpi u_0)^\circ+(q+1)\mathcal{Z}(u_0).
\end{equation}
Taking both identities \eqref{eq1} and \eqref{eq2} together, we obtain
\[
\mathcal{T}_1*\mathcal{Z}(u_0)=\CT_{\Gamma_0}*\mathcal{Z}(u_1)=\mathcal{Z}(\varpi u_0)^\circ+(q+1)\mathcal{Z}(u_0),
\]
which proves (1).

To prove (2), we see by the same argument  that
 $$
 \CT_{\Gamma_0}*\mathcal{Z}(\varpi^m u_0)^\circ=\mathcal{Z}(\varpi^mu_1)^\circ+q\mathcal{Z}(\varpi^{m-1}u_1)^\circ
 $$ and
 $$
\CT_{\Gamma_0}*\mathcal{Z}(\varpi^m u_1)^\circ=\mathcal{Z}(\varpi^{m+1}u_0)^\circ+q\mathcal{Z}(\varpi^{m}u_0)^\circ.
 $$ Thus
 $$
 \CT_1*\mathcal{Z}(\varpi^m u_0)^\circ=\CZ(\varpi^{m+1} u_0)^\circ+2q \CZ(\varpi^{m} u_0)^\circ+ q^2\CZ(\varpi^{m-1} u_0)^\circ,
 $$ as desired.
 \end{proof}

Let $\CZ$ and $\CZ'$ be formal schemes finite over a closed formal subscheme of codimension one in $\CN$,  with classes $[\CZ]\in K^{\CZ}(\CN)$, resp. $[\CZ']\in K^{\CZ'}(\CN)$. Let $\CT\to \CN\times\CN$ be a correspondence. Assume that $\Z\cap\CT(\CZ')$ is an artinian scheme (with support in the unique point of $\CN_\red$). Then, by \cite[Cor. 5.5]{LM}, we have  
\begin{equation}\label{compZ1}
\langle [\CZ], \CT_*([\CZ'])\rangle_{\CN}=\langle \CZ, \CT*\CZ'\rangle_{Z^1(\CN)} ,
\end{equation}
Here on the LHS appears the intersection product in K-theory (comp. \eqref{cupno}), and on the RHS appears the intersection number of properly intersecting cycles of codimension one on $\CN$, cf. \cite[Rem. 5.3]{LM}.

We now return to the AFL. Recall that $\Delta\subset \CN_1\times\CN_2$ is the graph of the closed immersion $\CN_1\hookrightarrow \CN_2$. Under the identification $\CN_1\times\CN_2\simeq\CN$, we can identify $\Delta$ with $\CZ(u_0)$.
\begin{corollary}\label{cor:int 1}
Let $m\geq 1$. Then  for all $g\in \U(\BV_2)$ regular semisimple,
$$\Int(g, \phi_m)=\langle g\Delta, (\CT^\circ_{m})_*(\Delta) \rangle_{\CN}=1.$$
\end{corollary}

\begin{proof}
By $\Delta=\CZ(u_0)$  and Proposition \ref{thm: T Del}, we have by \eqref{compZ1}
\begin{align*}
\Int(g, \phi_m)&=\langle g\Delta, \CT^\circ_{m}*\Delta \rangle_{Z^1(\CN)}=\langle \CZ(g\cdot u_0), \CZ(\varpi^m u_0)^\circ \rangle_{Z^1(\CN)}\\
&=\langle \CZ(g\cdot u_0), \CZ(\varpi^m u_0) \rangle_{Z^1(\CN)}-\langle \CZ(g\cdot u_0), \CZ(\varpi^{m-1} u_0) \rangle_{Z^1(\CN)}.\notag
\end{align*}
It remains to compute $\langle \CZ(g\cdot u_0), \CZ(\varpi^m u_0)) \rangle_{Z^1(\CN)}$ for all $m\geq 0$. We can use \cite{Kudla2011}.
The fundamental matrix of the lattice $\pair{g\cdot u_0, \varpi^m u_0 }$ for the obvious basis is
$$
A=\begin{pmatrix}
 1 & \varpi^m(g\cdot u_0,  u_0 )  \\
 \varpi^m( u_0,g\cdot u_0 )&  \varpi^{2m}
\end{pmatrix}
$$
and of the lattice $\pair{g\cdot u_0, \varpi^{m-1} u_0 }$ is
$$
A'=\begin{pmatrix}
 1 & \varpi^{m-1}(g\cdot u_0,  u_0 )  \\
 \varpi^{m-1}( u_0,g\cdot u_0 )&  \varpi^{2(m-1)}
\end{pmatrix} .
$$
Note that the valuation of $g\cdot u_0 $ is zero and the space $\BV$ is non-split (hence $\val(\det A)$ is odd). It follows that $(g\cdot u_0,  u_0 )$ is a unit, and the 
 fundamental invariants of $A$ (resp. $A'$) are $(0, \val(\det A))$ (resp. $(0,\val(\det A)-2)$). 
Then by \cite{Kudla2011} we obtain
$$
\langle \CZ(g\cdot u_0), \CZ(\varpi^m u_0)) \rangle_{Z^1(\CN)}=\frac{\val(\det(A))+1}{2}
$$
and
$$
\langle \CZ(g\cdot u_0), \CZ(\varpi^{m-1} u_0)) \rangle_{Z^1(\CN)}=\frac{\val(\det(A))-1}{2}.
$$
The assertion   follows.
\end{proof}

\begin{remark}
Note that $\BT_{\leq 2}$ is the composition of intertwining Hecke correspondences $\BT_2^+$ and $\BT_2^-$ which correspond to elements in the Iwahori Hecke algebra.  One may try to generalize the result above from the spherical Hecke algebra to the Iwahori Hecke algebra; we have not done so.
 \end{remark}

\subsection{Comparison}
\begin{theorem}\label{thm: n=1}
Conjecture \ref{AFL} holds when $n=1$, namely 
AFL  holds for the full Hecke algebra when $n=1$.
\end{theorem}

\begin{proof}
When $n=1$, the Hecke algebra $\CH_{K^\flat}$ is trivial. Hence it suffices to show the inhomogeneous AFL conjecture \ref{AFL inhom}. We show the identity  for $f$ running through the basis $\phi_m$ of $\CH_K$.

 By Lemma \ref{lem:BC m} (iii), we have $(\Bc^{\eta}_S)^{-1}(\phi_m)=\tilde\varphi'_m\in\CH_{K'_S}$ and hence we need to show that  for matching $g$ and $\gamma$,
$$\Int(g, \phi_m)\cdot\log q=-\omega(\gamma)\del(\gamma,\tilde\varphi'_{m}) .
$$ 

When $m\geq 1$,
this follows from directly comparing Proposition \ref{prop del O n=1}  (for orbital integrals) and Corollary \ref{cor:int 1} (for intersection numbers).
More precisely, without loss of generality, we let   $\gamma=  \begin{pmatrix}
 a & b\\c
& d
\end{pmatrix}\in S_2(F_0)_\rs$  with $v(c)=0$. Then it matches an element in $U(\BV_2)_\rs$ if and only if  $v(bc)=v(1-a\ov a)$ is odd (in particular $v(a)\geq 0$). 
Recall from \cite[\S2.4]{RSZ2} that the transfer  factor is  defined as
 $$
 \omega(\gamma)=(-1)^{v(\det(e,\gamma e))}, 
 $$
 where $e=(0,1)^t$ and hence $\det(e,\gamma e)=\det\begin{pmatrix}
 0 & b\\
1& d
\end{pmatrix}=-b$.
Therefore, when $v(b)=v(1-a\bar a)$ is odd, the transfer factor is $\omega(\gamma)=(-1)^{v(b)}=-1$ and  hence, by Proposition \ref{prop del O n=1}, (ii), 
 $$
 -\omega(\gamma)\del(\gamma,\tilde\varphi'_{m})=\log q.
  $$
  On the other hand, by Corollary \ref{cor:int 1} we have $$\Int(g, \phi_m)=1$$ for all $g\in U(\BV_2)_\rs$, which proves the desired identity. 
  
  The case $m=0$ is the AFL (for the unit element in Hecke algebra) proved in \cite{Zha12}.
We can also see this directly by using the facts proved in this section.   The proof of Corollary \ref{cor:int 1}  shows 
 \begin{align*}
\Int(g, \phi_0)=\langle g\Delta,  \Delta \rangle_{Z^1(\CN)}&=\langle \CZ(g\cdot u_0), \CZ( u_0) \rangle_{Z^1(\CN)}
\\&=\frac{\val(\det(A))+1}{2} ,
\end{align*}
where $$
A=\begin{pmatrix}
 1 & (g\cdot u_0,  u_0 )  \\
( u_0,g\cdot u_0 )&  1
\end{pmatrix}.
$$Note that $v(\det(A))=v(1-a\ov a)$, in terms of  the matrix form $g= \begin{pmatrix}
 a & b\\c
& d
\end{pmatrix}\in U(\BV_2)_\rs$ using the basis $\{u_0,u_1\}$.  

On the other hand, by Proposition \ref{prop del O n=1} (ii), for $\gamma=  \begin{pmatrix}
 a & b\\c
& d
\end{pmatrix}\in S_2(F_0)_\rs$ we have
$$
\del(\gamma,\tilde\varphi'_0)= \log q \begin{cases}  \frac{v(1-a\ov a)+1}{2}, & v(1-a\ov a)\geq 0\\
 0, & v(1-a\ov a)<0 .
 \end{cases}
 $$
 Matching $g$ and $\gamma$ have the same $a$.  The computation of the transfer factor is the same as in the case $m\geq 1$.  It follows that 
$$\Int(g, \phi_0)\cdot\log q=-\omega(\gamma)\del(\gamma,\tilde\varphi'_{0}) = \frac{\val(1-a\ov a)+1}{2}.
$$  
The proof is complete.

\end{proof}

\begin{remark}
The RZ space $\CN_2$ is isomorphic to the Lubin--Tate deformation space in height two, cf. Remark \ref{ex:dim2}. It is therefore natural to expect that the AFL for the full Hecke algebra is related to the linear AFL for the Hecke algebra of $\GL_2$ considered by Li \cite{Li}. We do not see an immediate passage between the two statements.
\end{remark}

\section{Base-point freeness}\label{s: del orb}
In this section, we comment on how much information on an element of the Hecke algebra is contained in either side of the AFL, i.e., in the functionals on the Hecke algebra defined by the derivative of  orbital integrals, resp. by the intersection numbers. It turns out that there is a surprising contrast between the functional defined by orbital integrals and that defined by the derivative of orbital integrals.

\subsection{Orbital integrals}
For orbital integrals, there is the following fact. 
\begin{proposition}\label{B-P}
The map
$$
\Orb: \CH_{K^{\prime \flat}\times K'}\to C^\infty(G'_{\rs})
$$
is injective.
\end{proposition}
\begin{proof}We use a theorem of Beuzart-Plessis \cite[Cor. 4.5.1]{BeP}:
{\em For any $f'\in C_c^\infty(G')$, the orbital integral $\Orb(-,f')$ vanishes at all regular semisimple elements if and only if $I_{\Pi}(f')=0$ for all  $\Pi \in \Phi_{\rm temp}(G'(F_0))$.} Here $\Phi_{\rm temp}(G'(F_0))$ denotes the set of irreducible tempered  
admissible representations  of $G'(F_0)$, and $I_{\Pi}(f')$ is the local relative character associated to the local Rankin--Selberg  period integral and the local Flicker--Rallis period integral.

Let $f'\in \CH_{K^{\prime \flat}\times K'}$ with identically vanishing orbital integrals. Then by the above quoted theorem, $I_{\Pi}(f')=0$ for all $\Pi$ as above. In particular, $I_{\Pi}(f')=0$ for all $\Pi\in \Phi^{\rm ur}_{\rm temp}(G'(F_0))$, the set of tempered unramified representations. However,  for 
$\Pi\in \Phi^{\rm ur}_{\rm temp}(G'(F_0))$, we have $I_{\Pi}(f')=\lambda_\Pi(f') I_\Pi({\bf 1}_{\K})$. Here $\lambda_\Pi: \CH_{K^{\prime \flat}\times K'}\to\BC$ is the algebra homomorphism associated to $\Pi$ or, equivalently, the evaluation of $\Sat(f')$ at the Satake parameter of $\Pi$. Note that $I_\Pi({\bf 1}_{\K})\neq 0$ for $\Pi\in \Phi^{\rm ur}_{\rm temp}(G'(F_0))$. Hence $\lambda_\Pi(f')=0$ for all $\Pi\in \Phi^{\rm ur}_{\rm temp}(G'(F_0))$. It follows that $f'=0$ by the Satake isomorphism.
\end{proof}
\subsection{Derivative of orbital integrals} 
Let $G'_{\rs, W_1}$ denote the open subset of $G'_{\rs}$ consisting of regular semisimple elements matching with elements in the non-quasi-split unitary group $G_{W_1}$.
\begin{conjecture}\label{conj:avoid}
The map
$$
\del: \CH_{K^{\prime \flat}\times K'}\to C^\infty(G'_{\rs, W_1})
$$
has a large kernel, in the sense that the kernel generates the whole ring $\CH_{K'}$ as an ideal (note that this kernel is only a vector subspace rather than an ideal). Similarly,  the map defined by the intersection numbers, $\Int:  \CH_{K^\flat\times K}\to C^\infty(G'_{\rs,W_1})$, has a large kernel.
\end{conjecture}

\begin{remark}A weaker conjecture would be that  ``the kernel is not contained in any maximal ideal of $\CH$ corresponding to a tempered representation". Here a $\BC$-point $\alpha$ of $\Spec\CH_{K'}$ is called tempered if, in terms of the coordinates in \S\ref{ss:Hk}, we have $|\alpha_i|=1$ for all $i$. There might be some other ways to formulate the smallness of the image. On the other hand, the image should not be too small, although we do not have a precise conjecture. In general it is unclear to us how to characterize the image. 
\end{remark}

\begin{theorem}\label{thm: n=1}
Conjecture \ref{conj:avoid} holds when $n=1$.
\end{theorem}

\begin{proof}

When $n=1$, the factors $\CH_{K'^\flat}$ and $\CH_{K^\flat}$ are trivial.   We note that
   the orbital integral map $\del$ factors through the  base change homomorphism $\Bc$. We now identify the set of orbits in $G'_{\rs}$
with the set of orbits in $S_{\rs}$. The induced map 
 \begin{align}
\del_G: \CH_{K}\to C^\infty(G'_{\rs, W_1})
\end{align}
can be written in terms of 
$$
\del_G(-,\phi ):=\omega_{S}(-)\del(-,(\Bc^{\eta}_S)^{-1}(\phi)),\quad \phi\in \CH_{K}.
$$  
Then the assertion is equivalent to the statement that the kernel of $\del_G$ generates the whole ring $ \CH_{K}$ as an ideal.
 We use the explicit results in Proposition \ref{prop del O n=1}, which shows that the image is exactly two dimensional, spanned by the image of $\phi_0$ and any one of the $\phi_m,\, m\geq 1$.  In particular, the kernel of  $\del_G$ is spanned by the set $\{\phi_m-\phi_1\mid m\geq 2\}$.  Equivalently it remains to show that the elements  of this set  have no common zero. From \eqref{Sat fm alt},  we have for $m\geq 2$\begin{equation*}
 \Sat(  \phi_m-\phi_1)=q^{m}\sum_{i=-m}^m X^{i}-q^{m-1}\sum_{i=-(m-1)}^{m-1} X^{i}- q(X+1+X^{-1})+1. 
 \end{equation*}
 It is straightforward to check that these Laurent polynomials have no common zero in $X\in\BC^\times$. Indeed, already these Laurent polynomials for $m=2$ and $m=3$ have no common zero. To see this, we can write these Laurent polynomials as polynomials $P_2$, resp. $P_3$, of degree $2$, resp. $3$, in $\BC[X_1]$, where $X_1=X+X^{-1}$. Long division shows that these polynomials are coprime, hence there are polynomials $R_2, R_3\in \BC[X_1]$ with $P_2R_2+P_3 R_3=1$. Rewriting this identity in terms of $X^{\pm1}$ shows that the Laurent polynomials for $m=2, 3$ have no common zero.
 The proof is complete.
  \end{proof}

 \begin{remark}A similar result in the setting of the linear AFL (for $\GL_2$ rather than $\U_2$) could be deduced from the result of Q. Li \cite[Prop. 7.6]{Li}.
\end{remark}

\section{Appendix: Correspondences for formal schemes and  maps on K-groups}
In this appendix, we explain the calculus of correspondences that we use, comp. also \cite[App. B]{Zha21}. Let $\breve O$ be a strictly henselian DVR. We consider  locally noetherian formal schemes,  locally formally of finite type over $\breve O$. For such a formal scheme $\CX$ and a closed formal subscheme $A$ of $\CX$, we denote by $K^A(\CX)$ the Grothendieck group of finite complexes of locally free modules over the structure sheaf which are formally acyclic outside $A$, comp. \cite{Gillet1987}. Recall from \cite[App. B]{Z14} that this means that the cohomology sheaves of $C$ are annihilated by a power of the ideal sheaf of $A$, locally on $\CX$. We similarly have the Grothendieck group $K'(A)$ formed by finite complexes of coherent modules on $A$. 

\subsection{Induced map on K-groups}\label{induced}
Let $\CX$ and $\CY$ be formal $\breve O$-schemes as above.  Let $\CT\to \CX\times_{\breve O} \CY$ be a \emph{geometric correspondence}, i.e., a formal scheme as specified above, with a morphism of formal schemes as indicated. We also simply write $\CT$ for this correspondence and denote by $p_\CX$, resp. $p_\CY$, the maps to $\CX$, resp. $\CY$. We assume that  $\CY$ is regular and that $p_\CY$ is a proper morphism. Then $\CT$ defines a map on K-groups with supports,
\begin{equation}\label{mapK}
\CT_*\colon K^A(\CX)\to K^{\CT(A)}(\CY), \quad x\mapsto Rp_{\CY, *}(p_\CX^*(x)) .
\end{equation}
Here $\CT(A)=p_\CY(p_\CX^{-1}(A))$, i.e., $\CT(A)$ is the image of the closed formal subscheme  $p_\CX^{-1}(A)$ of $\CT$ under the proper morphism $p_\CY$. This image is by definition the closed formal subscheme of $\CY$ defined by  the kernel ideal of the map $\CO_\CY\to p_{\CY, *}(\CO_\CT)$. Note here that by the properness of $p_\CY$, the target of this map is a coherent $\CO_\CY$-module. We note that $K^A(\CX)$ only depends on the radical of the ideal sheaf defining the closed formal subscheme $A$. Let us explain  the construction of the map. The map is the composition of the maps
\begin{equation}
p_\CX^*\colon K^A(\CX)\to K^{p_\CX^{-1}(A)}(\CT), \quad p_{\CY, *}\colon K^{p_\CX^{-1}(A)}(\CT)\to K^{\CT(A)}(\CY) .
\end{equation} 
For the first map, let $x\in K^A(\CX)$ be represented by a finite complex $C$ of locally free $\CO_\CX$-modules formally acyclic outside $A$. Then its base change to $\CT$ is a finite complex of locally free $\CO_\CT$-modules which is formally acyclic outside $p_\CX^{-1}(A)$, cf. \cite[\S 1.5]{Gillet1987}. The image $p_\CX^*(x)$ is defined to be the class of this complex.  The second map is the composition of three maps. First,   the natural map  $K^{p_\CX^{-1}(A)}(\CT)\to K'(p_\CX^{-1}(A))$, sending a complex $C$ which is formally acyclic outside a closed subset to the alternating sum of the classes in $K'$ of the cohomology sheaves of $C$. Second, the full direct image map $K'(p_\CX^{-1}(A))\to K'(p_\CY(\CT(A)))$, defined by  the properness of $p_\CY$.  Third, the identification   $K'(\CT(A))=K^{\CT(A)}(\CY)$ by the regularity of $\CY$, cf. \cite[Lem. 1.9]{Gillet1987}.

\subsection{Composition}\label{compofcorr}  Recall the composition of geometric correspondences. Let $\CT\to \CX\times_{\breve O} \CY$ and $\CS\to \CY\times _{\breve O} \CZ$ be geometric correspondences as above. The composition $\CU=\CS\circ\CT$ of these correspondences  $\CT$ and $\CS$  is  defined by the following diagram with cartesian square, 
$$\xymatrix{&&\CU \ar[rd]^{p'_\CY}  \ar[ld]_{q'_\CY} & \\ &\CT \ar[rd]_{p_\CY}  \ar[ld]_{p_\CX} & &\CS \ar[rd]^{q_\CZ}  \ar[ld]^{q_\CY} &\\ \CX &&  \CY&&  \CZ.}$$
In other words, the projections for $\CU$ are $r_\CX=p_\CX\circ q'_\CY$ and $r_\CZ=q_\CZ\circ p'_\CY$.

  Assume now that $\CY$ and $\CZ$ are regular and the morphisms $p_\CY$ and $q_\CZ$ proper, so that also $r_\CZ$ is proper. Let $A$ be a closed formal subscheme of $\CX$. Then the  three maps   are defined,
 $$\CT^A_*\colon K^A(\CX)\to K^{\CT(A)}(\CY), \quad\CS^{\CT(A)}_*\colon K^{\CT(A)}(\CY)\to K^{\CU(A)}(\CZ), \quad \CU^A_*\colon K^A(\CX)\to K^{\CU(A)}(\CZ) .
 $$
\begin{lemma}\label{complem}
Assume that the maps $p_\CY$ and $q_\CY$ are tor-independent, i.e., ${\mathcal Tor}_j^{\CO_\CY}(\CO_\CT, \CO_\CS)=0, \forall j>0$, cf. \cite[Def. 36.22.2]{SPA17}. Then 
\[\CU^A_*=\CS^{\CT(A)}_*\circ\CT^A_*\colon K^A(X)\to K^{\CU(A)}(Z) .
\]
\end{lemma}
\begin{proof}
We need to show that
$$
Rq_{\CZ, *}\big(Rp'_{\CY, *}({q'}_\CX^*(p_\CX^*(x)))\big)=Rq_{\CZ, *}\big(q_\CY^*(Rp_{\CY, *}(p_\CX^*(x)))\big) .
$$
Here $q_\CY^*(Rp_{\CY, *}(p_\CX^*(x)))$ makes sense as an element in $K(\CS)$ since,  by the regularity of $\CY$,  the complex $Rp_{\CY, *}(p_\CX^*(x))$ can be interpreted as an element in $K(\CY)$. 
It therefore suffices to prove the equality of elements in $K'(\CS)$,
$$
Rp'_{\CY, *}\big({q'}_\CX^*(p_\CX^*(x))\big)=q_\CY^*\big(Rp_{\CY, *}(p_\CX^*(x))\big) .
$$
This follows from base change for the cartesian square, again representing $Rp_{\CY, *}(p_\CX^*(x))$ by a finite complex of locally free $\CO_\CY$-modules, cf. \cite[Lem. 36.22.5]{SPA17}.   
 \end{proof}
\begin{remark}\label{dercomp}
We  use this lemma only in the case when   one of the two morphisms $p_\CY$ and $q_\CY$ is flat. In general, the identity $\CU^A_*=\CS^{\CT(A)}_*\circ\CT^A_*$ does not hold. However, it does always hold in the context of \emph{derived formal schemes}. Indeed, in this context, the base change formula used above always holds, comp. \cite[part III, ch. 3, Prop. 2.2.2]{GR}.  
\end{remark}

\section*{Declarations}

\subsection*{Data Availability Statements} Data sharing not applicable to this article as no datasets were generated or analysed during the current study.
\subsection*{Conflicts of interest/Competing interests} The authors have no conflicts of interest to declare that are relevant to the content of this article.
\subsection*{Code availability} Not applicable.
\subsection*{Funding}CL was supported by the NSF grant DMS \#2101157.  MR was supported by the Simons foundation. WZ was supported by the NSF grant DMS \#1901642 and the Simons Foundation.


\begin{thebibliography}{AB}


   \bibitem{BeP}
     Rapha\"el Beuzart-Plessis. \emph{Comparison of local relative characters and the Ichino–Ikeda conjecture for unitary groups}, J. Inst. Math. Jussieu (2020) 1–52.
     
   \bibitem{BP}
     Rapha\"el Beuzart-Plessis. \emph{A new proof of Jacquet-Rallis's fundamental lemma}, Duke Math. J. 170 (2021), no. 12, 2805--2814.

\bibitem{BPLZZ} Raphael Beuzart-Plessis, Yifeng Liu, Wei Zhang, and Xinwen Zhu. \emph{ Isolation of cuspidal spectrum, with application to the Gan-Gross-Prasad conjecture.} Ann. of Math. (2) 194 (2021), no. 2, 519--584.
 
 \bibitem {BW}
 Oliver B\"ultel and Torsten Wedhorn. \emph{Congruence relations for Shimura varieties associated to some unitary groups.} 
J. Inst. Math. Jussieu 5 (2006), no. 2, 229--261. 

   
  \bibitem{Di} Daniel Disegni. \emph{The $p$-adic Gross-Zagier formula on Shimura curves.} Compos. Math. 153 (2017), no. 10, 1987--2074.
  
  
    \bibitem{DZ}  Daniel Disegni and Wei Zhang. \emph{Gan--Gross--Prasad cycles and derivatives of $p$-adic $L$-functions}. In preparation. 
    
    \bibitem{FP} Najmuddin Fakhruddin and Vincent Pilloni, \emph{Hecke operators and the coherent cohomology of Shimura varieties}.  J. Inst. Math. Jussieu 22 (2023), no. 1, 1--69.
    
    \bibitem{FC} Gerd Faltings and   Ching-Li Chai. \emph{Degeneration of abelian varieties.} With an appendix by David Mumford. Ergebnisse der Mathematik und ihrer Grenzgebiete (3)  22. Springer-Verlag, Berlin, 1990. xii+316 pp.
    
 \bibitem{FYZ}   
Tony Feng, Zhiwei Yun, and  Wei Zhang. \emph{Higher theta series for unitary groups over function fields.} arXiv:2110.07001, Ann. Sci. ENS, to appear.

\bibitem{GR} Dennis, Gaitsgory, Nick Rosenblyum, \emph{A study in derived algebraic geometry, I. Correspondences and duality.} Mathematical Surveys and Monographs, 221. American Mathematical Society, Providence, RI, 2017. xl+533pp.

\bibitem{Gillet1987}  Henri Gillet and Christophe Soul\'e. \emph{Intersection theory using Adams operations}. Invent. Math. 90 (1987), no. 2, 243--277. 

 \bibitem{Go} Ulrich G\"ortz. \emph{ On the flatness of models of certain Shimura varieties of PEL-type.} Math. Ann. 321 (2001), no. 3, 689--727.

  \bibitem{GHR}
  Ulrich G\"ortz, Xuhua He, and Michael Rapoport. \emph{ Extremal cases of Rapoport-Zink spaces}. J. Inst. Math. Jussieu 21 (2022), no. 5, 1727--1782. 
  

   \bibitem{HLZ}
   Xuhua He, Chao  Li, and Yihang Zhu.\emph{ Fine Deligne-Lusztig varieties and arithmetic fundamental lemmas.} Forum Math. Sigma 7 (2019), e47, 55 pp.
   
  \bibitem{JL}  Herv\'e Jacquet and Stephen Rallis. \emph{ On the Gross-Prasad conjecture for unitary groups.} On certain L-functions, 205--264, Clay Math. Proc., 13, Amer. Math. Soc., Providence, RI, 2011. 
  
  \bibitem{Ko}
  Jean-Stefan Koskivirta.\emph{  Congruence relations for Shimura varieties associated with $GU(n-1,1)$.} Canad. J. Math. 66 (2014), no. 6, 1305--1326. 
  
\bibitem{Kudla2011}
Stephen~S. Kudla and Michael Rapoport.
 Special cycles on unitary Shimura varieties I. Unramified local
  theory. Invent. Math., 184(3) (2011):629--682.
  
  \bibitem{Laf} Vincent Lafforgue. \emph{Chtoukas pour les groupes r\'eductifs et param\'etrisation de Langlands globale.} JAMS 31 (2018), no. 3,  719--891.
  
  \bibitem{Lee} Si Ying Lee.  
\emph{Eichler-Shimura Relations for Shimura Varieties of Hodge Type.} arXiv:2006.11745 
 


  \bibitem{Les} Spencer Leslie.
  \emph{A fundamental lemma for the Hecke algebra: the Jacquet-Rallis case.} J. Number Theory, 243 (2023), 475--494.
  
  
  \bibitem{LRZ} Chao Li, Michael Rapoport and Wei Zhang. \emph{Quasi-canonical AFL and arithmetic transfer conjectures at parahoric levels.} arXiv:2404.02214. 
  
  \bibitem{Li}Qirui  Li.   \emph{An intersection formula for CM cycles on Lubin-Tate spaces.} Duke Math. J. 171 (2022), no. 9, 1923--2011.
  
  
  \bibitem{LM} Qirui Li and Andreas Mihatsch.
  \emph{On the linear AFL: the non-basic case}. 
  arXiv:2208.10144
  
\bibitem{LTXZZ}
Yifeng Liu, Yichao Tian, Liang Xiao, Wei Zhang, and Xinwen Zhu.
 \emph{ On the Beilinson-Bloch-Kato conjecture for Rankin-Selberg motives}, Invent. Math. 228 (2022), no. 1, 107--375.
 
\bibitem{Mihatsch2016}
Andreas Mihatsch.
 \emph{ Relative unitary RZ-spaces and the Arithmetic Fundamental Lemma}.
J. Inst. Math. Jussieu, 21 (2022), no. 1, 241--301. 

\bibitem{MZ}
Andreas Mihatsch and Wei Zhang. 
\emph{On the Arithmetic Fundamental Lemma conjecture over a general $p$-adic field.} arXiv:2104.02779. JEMS, to appear.

\bibitem{Off}
Omer Offen, \emph{Relative spherical functions on $p$-adic symmetric spaces (three cases)}, Pac. J. Math. 215 (1) (2004) 97--149.

\bibitem{Pil}
Vincent Pilloni. \emph{Higher coherent cohomology and $p$-adic modular forms of singular weights.} Duke Math. J. 169 (2020), no. 9, 1647--1807. 

 \bibitem{Ram} Dinakar Ramakrishnan. \emph{ A mild Tchebotarev theorem for GL(n).} J. Number Theory 146 (2015), 519--533. 
 
\bibitem{RSZ1}
Michael Rapoport, Brian Smithling, and Wei Zhang.
 \emph{On the arithmetic transfer conjecture for exotic smooth formal moduli
  spaces.}
 Duke Math. J., 166(12):2183--2336, 2017.
 
\bibitem{RSZ2}
Michael Rapoport, Brian Smithling, and Wei Zhang.
 \emph{Regular formal moduli spaces and arithmetic transfer conjectures.}
Math. Ann., 370(3-4):1079--1175, 2018.



\bibitem{RZ96}
Michael Rapoport and Thomas Zink.
  \emph{Period spaces for $p$-divisible groups}, volume 141 of 
  Annals of Mathematics Studies.
 Princeton University Press, Princeton, NJ, 1996.
 
 \bibitem{SPA17} The Stacks Project Authors. \emph{Stacks Project.} http://stacks.math.columbia.edu.
 
 \bibitem{Ter} Ulrich Terstiege, \emph{On the regularity of special difference divisors}. C. R. Math. Acad. Sci. Paris 351 (2013), no. 3-4, 107--109. 
 
 \bibitem{Wed}
  Torsten Wedhorn.\emph{ Congruence relations on some Shimura varieties.} J. Reine Angew. Math. 524 (2000), 43--71. 
  
\bibitem{XZ}Liang Xiao and Xinwen Zhu. \emph{Cycles on Shimura varieties via geometric Satake. Examples}. Addendum to arXiv:1707.05700
 
 \bibitem{Yun}
 Zhiwei Yun.  \emph{The fundamental lemma of Jacquet and Rallis}, With an appendix by Julia Gordon. Duke Math. J. 156 (2011), no. 2, 167--227.
 
  \bibitem{YZ} Zhiwei Yun, Wei Zhang.
 \emph{Shtukas and the Taylor expansion of L-functions},  Ann. of Math., Vol. 186 (2017), no. 3, 767--911.
 
 

  \bibitem{YZ2} Zhiwei Yun, Wei Zhang.
 \emph{Shtukas and the Taylor expansion of L-functions (II)}. Ann. of Math. Vol. 189 (2019), no. 2, 393--526.
 
 
\bibitem{Zha12}
Wei Zhang. \emph{On arithmetic fundamental lemmas.} Invent. Math. 188 (2012), no. 1, 197--252.

 \bibitem{Z14}
 Wei Zhang. \emph{Fourier transform and the global Gan-Gross-Prasad conjecture for unitary groups.} Ann. of Math. (2) 180 (2014), no. 3, 971--1049.

 \bibitem{Zha21} Wei Zhang. \emph{Weil representation and arithmetic fundamental lemma.} Ann. of Math. (2) 193 (2021), no. 3, 863--978. 
 
 \bibitem{ZZha} Zhiyu Zhang. 
\emph{Maximal parahoric arithmetic transfers, resolutions and modularity.} arXiv:2112.11994. Duke Math. J., to appear.


\end{thebibliography}
\end{document}